\documentclass[11pt,reqno,a4paper]{article}
\usepackage{amsmath}
\usepackage{amsbsy}
\usepackage{amsthm}
\usepackage{amssymb}
\usepackage{amscd}
\usepackage{mathabx}
\usepackage{enumitem}
\usepackage{accents}
\usepackage{float}
\usepackage{url}
\usepackage{tikz}
\usetikzlibrary{matrix,arrows,decorations.pathmorphing}
\usepackage{mathabx}

\usepackage{hyperref}
\hypersetup{
pdftitle={},%
pdfauthor={},%
pdfsubject={},%
pdfkeywords={},%
colorlinks=true,%
linkcolor={black},%
linktoc={},
linktocpage={},%
pageanchor={},
citecolor={black},
}

\usepackage[english]{babel}
\usepackage[utf8]{inputenc}
\usepackage[margin=2.41cm]{geometry}
\usepackage{mathrsfs}
\usepackage{hyperref}
\usepackage{url}

\usepackage{cite}

\usepackage{etoolbox}

\makeatletter
\patchcmd{\ttlh@hang}{\parindent\z@}{\parindent\z@\leavevmode}{}{}
\patchcmd{\ttlh@hang}{\noindent}{}{}{}
\makeatother

\newtheorem{thm}{Theorem}[section]
\newtheorem{cor}[thm]{Corollary}
\newtheorem{lemma}[thm]{Lemma}

\newtheorem{prop}[thm]{Proposition}

\newtheorem{definition}[thm]{Definition}

\newtheorem{conj}[thm]{Conjecture}

\theoremstyle{remark}

\theoremstyle{definition}

\newtheorem{rmk}[thm]{Remark}

\newtheorem{defn}[thm]{Definition}

 \renewcommand{\tilde}{\widetilde}

\newcommand{\vol}{{\textnormal{vol}_g\,}}

\newcommand{\voll}{{\textnormal{vol}_g}}

\renewcommand{\div}{{\textnormal{div}\,}}

\newcommand{\spt}{{\textnormal{supp}\,}}
\newcommand{\scal}{{\textnormal{scal}}}
\newcommand{\Ric}{{\textnormal{Ric}}}
\newcommand{\ric}{{\textnormal{ric}}}
\newcommand{\tr}{{\textnormal{tr}}}

\newcommand{\R}{{\mathbb{R}}}

\renewcommand{\S}{{\mathcal{S}}}
\newcommand{\T}{{\mathcal{T}}}

\newcommand{\G}{\mathbb{G}}

\newcommand{\M}{{\mathcal{M}}}
\newcommand{\Prob}{{\mathcal{P}}}
\newcommand{\U}{{\mathcal{U}}}

\def\eps{\varepsilon}
\def\Ch{\mathcal{C}_c}
\def\dd{{\rm d}}
\newcommand{\ro}{\varrho}

\def\ve{\eps}

\def\cU{\mathcal{U}}
\def\cN{\mathcal{N}}
\def\cV{\mathcal{V}}

\def\q{\mathbf{q}}

\def\proofof#1{\begin{proof}[Proof of #1]}

\def\part#1#2{\par\noindent{\underline{\it Part~#1.}}\emph{ #2}\\}

\begin{document}

\begin{flushright}

\baselineskip=4pt

\end{flushright}

\begin{center}
\vspace{5mm}

{\Large\bf THE FIVE GRADIENTS INEQUALITY \\[2mm]
ON DIFFERENTIABLE MANIFOLDS\\}
\


\vspace{5mm}

{\bf by}

\vspace{5mm}
{ \bf  Simone Di Marino$^{1}$, Simone Murro$^2$ and Emanuela Radici$^3$}\\[2mm]
\noindent  {\it $^1$ Universit\`a di Genova, Dipartimento di Matematica, MaLGa,\\ 
 Via Dodecaneso 35, I-16146 Genova (GE), Italy}\\[1mm]
 \noindent  {\it $^2$Universit\`a di Genova, Dipartimento di Matematica \& INFN,\\  Via Dodecaneso 35, I-16146 Genova (GE), Italy}\\[1mm]
\noindent  {\it $^3 $ Universit\`a degli studi dell'Aquila,  DISIM, \\
Via Vetoio 1,  67100 (Coppito) L'Aquila (AQ), Italy}\\[2mm]
email: \ {\tt simone.dimarino@unige.it}\,, \ {\tt murro@dima.unige.it}\,, \ {\tt  emanuela.radici@univaq.it}
\\[8mm]
\end{center}

\begin{abstract}
The goal of this paper is to derive the so-called five gradients inequality for optimal transport theory for general cost functions on two class of differentiable manifolds: locally compact Lie groups and compact Riemannian manifolds.
\bigskip

\begin{center}
\bf Résumé
\end{center}
Le but de cet article est de déterminer l'inégalité dite ``five gradients inequality'' pour la théorie du transport optimal avec une fonction coût générique, sur deux classes de variétés différentielles: les groupes de Lie localement compacts et les variétés riemanniennes compactes.
\end{abstract}

\\[0.5mm]

\tableofcontents

\renewcommand{\thefootnote}{\arabic{footnote}}
\setcounter{footnote}{0}

\section{Introduction}

Among variational problems involving optimal transportation and Wasserstein distances, it has been shown in~\cite{5gi} that the so-called \emph{`five gradients inequality'} plays a distinguished role. Indeed, it has been used to derive BV and Sobolev estimates for the solutions of the JKO scheme for diffusion evolution equations in $\mathbb{R}^n$. This is the case for example of nonlinear diffusion~\cite{5gi}, weighted ultrafast diffusion~\cite{Iacobelli} and Fokker-Planck or Keller-Segel equations for chemiotaxis~\cite{Simo5gi}.  It  has been also used to provide bounds on the perimeters of solutions of some variational problems involving mutually singular measures~\cite{Buttazzo}. Remarkably the five gradients inequality allowed also us to establish stronger convergence estimates for the JKO scheme for the Fokker-Planck equations~\cite{SanTosh}.

From a physical perspective, working in $\R^n$ and with the quadratic cost function given by $c(x,y)=|x-y|^2$ is quite restrictive: a recent generalization for different costs is presented in \cite{Caillet}, however the setting is still the Euclidean space.
Many models inspired by physics and biology involve diffusions on interfaces~\cite{Christensen,Debbasch,Sbalzarini,smerlak}
and it is convenient to model them using Riemannian manifolds. 
It is then natural  to extend the results in~\cite{5gi,Caillet} to differentiable manifolds and the first main result of this paper is a derivation of the five gradients inequality for any compact, complete and smooth (actually $C^3$) Riemannian manifold for any cost function $c(x,y)=h(d(x,y))$ where $h$ is a convex function. 

 In the sequel, every probability measure $\mu \in \Prob(\M)$ will be always absolutely continuous with respect to the volume measure $\vol$; we will often, by a slight abuse of notation, identify measures with their density with respect to $\vol$.

\medskip

\begin{thm}\label{thm:main} Let $(\M,g)$ be a smooth compact manifold with Ricci curvature bounded from below by $K \in \R$. 
Let $d$ denotes the Riemannian distance, $h \in C^1([0,\infty))$ be some nonnegative strictly convex increasing function such that $h'(0)=0$  and assume that the cost function $c: \M\times \M \to \R$ is defined by 
$c(x,y):= h(d(x,y))$. 
 Finally, let $\ell $ be an increasing, convex, isotropic function on the tangent space such that $\ell(0)=0$ and let be $\mu, \nu \in W^{1,1}(\M) \cap \Prob(\M)$.
 \\ Then, denoting by $\gamma$, $\phi$ and $\psi$ respectively the optimal plan and the optimal Kantorovich potentials for the optimal transport problem between $\mu$ and $\nu$ for the cost $c$, it holds
\begin{equation}\label{eqn:5GIW11}
\int_{\M}  \Bigl(  \ell' (\nabla \phi) \cdot \nabla \mu +   \ell' (\nabla \psi) \cdot \nabla \nu  \Bigr) \vol  \geq K \int_{\M \times \M} \ell'(h'(d(x,y)))d(x,y) \, d \gamma.
\end{equation}
\end{thm}

The compactness hypothesis on $(\M,g)$ is not really necessary, as it can be replaced by imposing that $\mu$ and $\nu$ have compact support: in this way we can consider also spaces with constant negative curvature, which would be otherwise not allowed by the completeness assumption.

\begin{rmk}\label{rmk:ell}
Notice that $\ell' : [0,\infty) \to [0, \infty)$. In \eqref{eqn:5GIW11} we use $\ell'$ on a vector, which acts by modifying its distance from the origin, keeping the same direction, namely:
$$ \ell'(v)= \begin{cases} \ell'(\|v\|) \frac {v}{\|v\|}  \quad &\text{ if }v \neq 0 \\ 0 & \text{ otherwise.} \end{cases}$$
It would be interesting to have a generalization, that is a monotone function acting on vectors: however while $\ell'(v)$ acts isotropically, so we can canonically identify it on every tangent space, it is not clear how a generic monotone function should change from point to point. We will see that in fact when a canonical identification of tangent spaces exists (for example in Lie groups), a more general monotone function can be used (see Theorem \ref{thm:5gi Lie Group}).
\end{rmk}

\begin{rmk}\label{rmk:contraction}
It is not surprising that Ricci bounds for the curvature appear in the formula:  in fact it is to be expected since we can show that inequality \eqref{eqn:5GIW11} for $\ell(t)=h(t)=t^2/2$ for $K\geq 0$ yields $K$-contractivity of the $2$-Wasserstein distance along the heat flow, which in turn it is known to be equivalent to having the Ricci curvature bounded below by $K$ (see  \cite{stvr}). In fact let $\mu_0, \nu_0 \in  W^{1,1}(\M) \cap \Prob(\M)$, and let $\partial_t \mu_t= \Delta \mu_t$ and $\partial_t \nu_t= \Delta \nu_t$. Then $\mu_t, \nu_t \in W^{1,1}(\M) \cap \Prob(\M)$; let $\phi_t , \psi_t$ be the Kantorovich potentials relative to $\mu_t,\nu_t$ for $c(x,y)=d(x,y)^2/2$. Then by the first variation formula of the Wasserstein distance we obtain
\begin{align*} \frac d{dt} \frac 12 W_2^2(\mu_t, \nu_t) &= \int_{\M} \phi_t \partial_t \mu_t  \, d \vol +  \int_{\M} \psi_t \partial_t \nu_t  \, d \vol \\
& = \int_{\M} \phi_t \Delta \mu_t  \, d \vol +  \int_{\M} \psi_t \Delta \nu_t  \, d \vol  \\
& = - \int_{\M}  \nabla \phi_t  \cdot\nabla \mu_t  \, d \vol -   \int_{\M} \nabla \psi_t \cdot \nabla \nu_t  \, d \vol  \\
& \leq -K \int_{\M \times \M}  d^2(x,y) \, d \gamma_t = -K W_2^2(\mu_t,\nu_t); 
\end{align*}
by Gr\"onwall we eventually get $W_2(\mu_t,\nu_t) \leq e^{-Kt} W_2(\mu_0, \nu_0)$.
\end{rmk}

\medskip
A separate result holds in $\R^n$ and more generally on Lie groups $\G$, which can be of independent interest: in those cases, we obtain the inequality for any cost $c(x,y)$ which is $L$-Lipschitz and translation invariant (without requiring any convexity assumptions). The proof in this case is rather different and it is more inspired by the ``regularity by duality'' approach, very different from the one introduced in~\cite{5gi,Simo5gi,Caillet}, which is the reason why we make less restrictive assumptions.

We denote by $\mathfrak{g}$ the Lie algebra of the Lie group, that is the vector space of right-invariant vector fields on $\G$: this is naturally isomorphic to $T_e \G$ and it has an antisymmetric multiplication given by the Lie bracket. Given a subspace $\mathfrak{X} \subseteq \mathfrak{g}$, we say that it is a \emph{generating subspace} if the algebra generated by it is the whole $\mathfrak{g}$. Sometimes these are called also \emph{horizontal distributions}, and they have the following crucial property: when we define a control distance associated to this distribution (called Carnot-Caratheodory distance), obtained by minimizing the length of curves whose derivatives belong at each point to the corresponding \emph{horizontal subspace}, then this distance is finite (as long as $\G$ is connected) for every pair of points and moreover generate the same topology as the Riemannian distance.

Then a function $f: \G \to \R$ is called H-Lipschitz if it is Lipschitz with respect to the Carnot-Caratheodory distance, and the horizontal gradient $\nabla_H f$ is representing the restriction to the distribution $\mathfrak{X}$ of the differential of $f$: given an orthonormal basis $X_1, \ldots, X_m $ of $\mathfrak{X}$ we can identify $\nabla_H f = (X_1 f, \ldots, X_m f)$.

 \begin{thm}\label{thm:5gi Lie Group} Let $\G$ be a connected Lie group with Lie algebra $\mathfrak{g}$ and let $\mathfrak{X} \subseteq \mathfrak{g}$ be a right invariant generating distribution of dimension $m$ for $\mathfrak{g}$. Let us consider on $\G$ the control distance relative to $\mathfrak{X}$, $d(x,y)=d_{\mathfrak{X}}(x,y)$; let $h: \G \to [0, \infty)$ be a $L$-Lipschitz continuous function (with respect to $d$), and let $\mu, \nu \in W^{1,1}_H(\G)$  and let $G: \mathbb{R}^m \to [0,+\infty)$ be a convex function of the form $G(x)= \int_{S^{m-1}} f_v( v \cdot x ) \, d\sigma (v) $ for a family of even convex functions $f_v$ and a positive measure $\sigma$ on $S^{m-1}$.

Let  $\phi, \psi$ two optimal $L$-Lipschitz Kantorovich potentials for the optimal transport problem between $\mu$ and $\nu$ with the cost $c(x,y)=h (y^{-1} x)$. Then 
\begin{equation}\label{eqn:5giG} - \int_{\G} \Big( \nabla G ( \nabla_H\phi   ) \cdot \nabla_H \mu  + \nabla G ( \nabla_H \psi  ) \cdot \nabla_H \nu \Big) \, dx \leq 0 \,. 
\end{equation}

For any horizontal right-invariant vector field $X \in \mathfrak{X}$ we have $X=v \cdot (X_1, \ldots, X_m)$ for some $v$; considering the  special case of \eqref{eqn:5giG} with $G(w) = f ( w \cdot v)$, where $f\in C^1(\R)$ be an even convex function we obtain
\begin{equation}\label{eqn:5gif}
- \int_{\G} \Big( f' ( X (\phi)  ) X \mu  + f' ( X(\psi)  ) X \nu \Big) \, dx \leq 0 \,.
\end{equation}

 \end{thm}

 Notice that in particular \eqref{eqn:5giG} holds for the notable cases $G(x)=\|x\|^p$ for every $p > 1$: it is sufficient to choose $\sigma$ the uniform measure on $S^{m-1}$ and $f_v(t)=C_p|t|^p$ for some constant $C_p>0$. Notice that $G$ plays the same role of the function $\ell$ in Theorem~\ref{thm:main}, but thanks to the richer geometry allowed by the Lie structure, we can drop the isotropy assumption. Indeed in this setting we can even prove a five gradients inequality for directional derivatives (see \eqref{eqn:5gif}).

As a consequence of Theorem~\ref{thm:5gi Lie Group}, it is clear that in Theorem~\ref{thm:main} assuming the Riemannian manifold to be compact is not a  necessary assumption. On the contrary, having Ricci curvature bounded from below seems to be a necessary condition. The class of manifold that `behave' similarly to compact manifolds is the so-called \textit{Riemannian manifold of bounded geometry}. Let us recall that a Riemannian
manifold $(\M,g)$ is of bounded geometry if $g$ is complete and the curvature
tensor and all its covariant derivatives are bounded. This leads to the following conjecture.
\begin{conj}
A complete Riemannian manifold is of bounded geometry if and only if the five gradients inequality holds.
\end{conj}

Notice that due to Remark~\ref{rmk:contraction}, thanks to the results of the seminal paper by Sturm and von Renesse \cite{stvr}, we have that the double implication holds at least for compact Riemannian manifolds. However the equivalence between the contraction in the Wasserstein distance and Ricci curvature goes much further: Savaré proved in \cite[Theorem 4.1]{Savare}  that contraction of the Wasserstein distance along the heat flow implies Ricci curvature bounded from below in the abstract setting of infinitesimally Hilbertian metric spaces with synthetic Ricci curvature bounded from below, introduced in \cite{gigli2015}. We thus believe also the following to be true.
\begin{conj}
Let $(X,d, \mu)$ be an infinitesimally Hilbertian Polish metric space. Then $(X,d,\mu)$ is $RCD(K, \infty)$ if and only if the five gradients inequality holds.
\end{conj}

\subsection{Sketch of the proof in $\R^n$}\label{ss:rn}

We present here the idea at the core of the proof of the ``five gradients inequality'' in the settings that we consider in the present work. Differently from the explicit approach of \cite{Caillet, 5gi}, we deduce the inequality as a consequence of the optimality of the Kantorovich potentials in the dual formulation of the optimal transport problem for a Lipschitz cost function: moreover this approach does not require any a priori smoothness for $\mu, \nu$ or the potentials.

For simplicity, we explain our strategy directly in $\R^n$. Let then $\mu,\nu \in W^{1,1}(\R^n)$ be compactly supported probability measures on $\R^n$, $\phi,\psi$ be the Lipschitz Kantorovich potentials for the classical quadratic Wasserstein distance between $\mu$ and $\nu$ and $T: \R^n \to \R^n$ the optimal transport map. Then the couple $(\phi,\psi)$ solves 
\[   \int_{\R^n} \phi d\mu + \int_{\R^n} \psi d\nu = \sup_{(\tilde{\phi}, \tilde{\psi})} \left\lbrace   \int_{\R^n} \tilde{\phi} d\mu + \int_{\R^n} \tilde{\psi} d\nu \; : \;\, \tilde{\phi}(x) + \tilde{\psi}(y) \leq \frac{|x-y|^2}{2} \right \rbrace,    \]
moreover
\[ \phi(x) + \psi(T(x)) = \frac{|x-T(x)|^2}{2}  \quad \text{and} \quad \nabla \phi(x) = x - T(x) \;\;\; \mu\text{-a.e.}   \]

Let us first describe the linear case, namely when $\nabla G$ in \eqref{eqn:5giG} is the identity map and the ``five gradients inequality'' actually becomes a ``four gradients inequality''. 
Observe, that the couple $(\phi_h,\psi_h)$ defined as
\[ \phi_h(x) = \phi(x + h e) , \qquad \psi_h(x) = \psi(x + h e)   \]
for some unit vector $e \in \mathbb{S}^{n-1}$ and $h >0$, is a good competitor for the dual problem because we still have $\phi_h(x) + \psi_h(y) \leq |x-y|^2/2$. Moreover, if we set 
\[ \mathcal{E}(h) : = \int_{\R^n} \phi_h d\mu + \int_{\R^n} \psi_h d\nu,    \] 
then from the optimality of the Kantorovich potentials we deduce that $\mathcal{E}$ has a maximum point in $h=0$ thus 
\[
\mathcal{E}(h) + \mathcal{E}(-h) - 2\mathcal{E}(0) \leq 0.
\]
The integral on $\R^n$  being translation invariant, the previous inequality implies that
\[ - \int_{\R^n} (\phi_h - \phi) d(\mu_h - \mu) - \int_{\R^n} (\psi_h - \psi) d(\nu_h - \nu) \leq 0,  \]
then dividing by $h$, taking the limit as $h \searrow 0$ and summing up over an orthonormal basis of $\R^n$ one obtains 
\[ - \int_{\R^n} \nabla \phi \cdot \nabla \mu dx -  \int_{\R^n} \nabla \psi \cdot \nabla \nu dx \leq 0. \]

Let us now consider a convex function $G : \R^n \to \R$ as in Theorem~\ref{thm:5gi Lie Group}. Let $e \in \mathbb{S}^{n-1}$ be fixed, then the couple $(\tilde{\phi}_h, \tilde{\psi}_h)$ defined by 
\[ \tilde{\phi}_h(x) = \phi(x) + h f'_e\left(\frac{\phi_h(x) - \phi(x)}{h} \right), \qquad \tilde{\psi}_h(x) = \psi(x) + h  f'_e\left(\frac{\psi_h(x) - \psi(x)}{h} \right)   \]
does not satisfy anymore the constraint that $\tilde{\phi}_h(x) + \tilde{\psi}_h(y) \leq |x-y|^2/2$ for every $x,y \in \R^n$. However the crucial observation is that this still holds in the support of any optimal plan $\gamma$. In particular
\[  \int_{\R^n} \tilde{\phi}_{\pm h} d\mu + \int_{\R^n} \tilde{\psi}_{\pm h} d\nu = \int_{\R^n \times \R^n} \big(\tilde{\phi}_{\pm h}(x) + \tilde{\psi}_{\pm h}(y) \big) d\gamma  \leq  \int_{\R^n} \phi d\mu + \int_{\R^n} \psi d\nu.  \]
Similar computations as in the linear case lead to
\[ - h \int_{\R^n} (\tilde{\phi}_{h} - \phi) d(\mu_h - \mu) -  h \int_{\R^n} (\tilde{\psi}_{h} - \psi) d(\nu_h - \nu) \leq 0   \]
and hence to 
\[ - \int_{\R^n} f'_e(\nabla \phi \cdot e) \nabla \mu \cdot e \,dx - \int_{\R^n} f'_e(\nabla \psi \cdot e) \nabla \nu \cdot e \,dx  \leq 0 \]
for the fixed direction $e$ and then a final integration over $e \in \mathbb{S}^{n-1}$ provides
\[   -\int_{\R^n} \nabla G(\nabla \phi) \cdot\nabla \mu \,dx - \int_{\R^n} \nabla G(\nabla \psi) \cdot \nabla \nu \, dx \leq 0.  \]

In the present work we apply the above strategy in the context of locally compact Lie groups and compact Riemannian manifolds. The first setting does not present substantial differences with respect to the Euclidean case $\R^n$, as it is possible to find a point-independent orthonormal basis of the space and the Lebesgue measure is invariant under translation. The Riemannian case is instead more complicated and one has to argue locally. For the linear case, in essence in the above sketch we used that if $\phi(x)+\psi(y) \leq |x-y|^2/2$ then we have
$$ \Delta \phi (x) + \Delta \psi(y) \leq 0 \qquad \text{ on the contact set }\phi(x)+\psi(y) = |x-y|^2;$$
this inequality is obtained by using a simultaneous parallel variation in $x$ and $y$, which do not alter the distance between them, and then average in the directions.

In the Riemannian setting in order to do a simultaneous variation in $x$ and $y$, we perform a flow along the geodesic connecting them, similar to the one generated by Fermi coordinates in \cite[Section $3$]{Andrews}; now $d(x_t,y_t)^2$  is not constant anymore and we will see in Section~\ref{sec:simultaneous} that in its second variation, the Ricci curvature naturally appears. Notice that in \cite{Andrews} sharper estimates are obtained for $ \frac {d^2}{dt^2} d(x_t,y_t)^2$: we expect, as a consequence, that some finer estimates could be obtained also in our case, but we did not pursue them for the sake of simplicity (i.e. linear dependence in $K$ of the right-hand side).

\subsection{Applications} \label{applications}

We conclude this introduction commenting on the $BV$ regularity of certain variational problems in Optimal Transportation on Riemannian manifolds. Indeed, a direct application of our main Theorem \ref{thm:main} provides explicit estimates for the Wasserstein projection of measures with $BV$ densities over the set 
$\{ \rho \in \Prob(\M) : \rho \ll \vol \text{ and }   \frac{d \rho}{d \vol}   \leq f \} $
in terms of the $BV$ norm of $f$. This problem is of particular interest for the choice $f=1$ as it corresponds to the case of projection on densities with $L^\infty$ bound, thus having a natural application in the context of evolutionary PDEs describing crowd motion.     
The Euclidean version of this result has been proved in  \cite[Theorem 1.1]{5gi} for $f=1$ and in  \cite[Theorem 1.2]{5gi} for generic $f$ with mass bigger than $1$, the key ingredient of both proofs being the \emph{`five gradients inequality'} in $\R^n$. 
Thanks to McCann polar decomposition theorem on Riemannian manifolds \cite{MC} and straightforward computations in local charts, the same argument provided in \cite{5gi} still works on smooth compact Riemannian manifolds and allows us to apply the Riemannian version of the 
\emph{`five gradients inequality'}, i.e. estimate \eqref{eqn:5GIW11}. 

More precisely, we obtain the following.

\begin{thm}\label{thm:BVestimates}
Let $(\M,g)$, $K$, $d$, $h$ be as in Theorem \ref{thm:main} and $\ell(t)=t$. Let then $\eta : \R_+ \to \R \cup \{+\infty\}$ be a convex and l.s.c. function and $\nu \in \Prob
(\M) \cap BV(\M)$ be a given measure. If $\bar \mu$ is a solution of 
\[ \min_{\mu\, \in\, \Prob (\M)} \; \Ch (\mu,\nu) + \int_\M \eta(\mu) \vol  \]
then
\begin{equation}\label{eqn:BVcontraction}
\int_\M |\nabla \bar \mu| \vol + K \int_{\M \times \M} d(x,y) d \gamma  \leq \int_\M |\nabla \nu| \vol. 
\end{equation}
where $\gamma$ is an optimal plan between $\bar{\mu}$ and $\nu$ for the cost $c = h \circ d$. 
Moreover, if $f \in BV(\M)$ is a non negative function with mass $\int_\M f \vol \geq 1$ and
\[ \bar \mu \in {\rm argmin} \{ \Ch (\mu,\nu) : \mu \in \Prob(\M),\, \mu \leq f \; \text{$\vol$-a.e.} \}, \]
then 
\begin{equation}\label{eqn:BVprojection}
\int_\M |\nabla \bar \mu| \vol + K \int_{\M \times \M} d(x,y) d \gamma  \leq \int_\M |\nabla \nu| \vol + 2 \int_\M |\nabla f| \vol
\end{equation}
where, again, $\gamma$ is an optimal plan between $\bar{\mu}$ and $\nu$ for $c= h \circ d$.
\end{thm}

Estimates \eqref{eqn:BVcontraction} and \eqref{eqn:BVprojection} then have some important consequences depending on the sign of the Ricci curvature. Indeed, recalling that $\gamma$ is an optimal plan for the cost $c=h \circ d$ with marginals $\bar \mu$ and $\nu$, one has that the two estimates provide an explicit control of the change of the BV norm in terms of either the $1$-Wasserstein distance  or (via Jensen inequality), a suitable homogenized version of the optimal cost $\Ch$  between the two measures $\bar \mu$ and $\nu$, depending on the sign of $K$. Indeed, one has 
\[   -  K \int_{\M \times \M} d(x,y) d \gamma \leq \begin{cases}  -K W_1(\bar{\mu},\nu)  \qquad & \text{ if }K>0 \\  -K \cdot h^{-1}  \bigl( \Ch (\bar{\mu},\nu) \bigr) & \text{ if }K<0. \end{cases} \]
Notice that in the application of Theorem~\ref{thm:BVestimates} to the JKO scheme one has $h(t)=t^2$ and so this last estimate, for $K<0$ reduces to $-K W_2(\bar{\mu},\nu)$.

\subsection{Possible generalizations}\label{ss:general}

We want also to comment on five gradient inequalities holding for other distance-like functions. Notice that in fact in Subsection~\ref{ss:rn} the relevant facts of the squared Wasserstein distance that we use are
\begin{itemize}
\item there exists a dual formulation, which depends \emph{linearly} on the dual variables;
\item the admissible set of dual potential is \emph{invariant under translation}.
\item the optimal dual potentials are differentiable almost everywhere.
\end{itemize}
This of course let us immediately realize that the sketch we gave for $|x-y|^2$ works also with $\ell(x-y)$, where $\ell$ is any positive Lipschitz cost (this is what we do in Theorem \ref{thm:5gi Lie Group} in the case $\G=\R^n$). However we can change also the type of functional. Notable examples are:

\begin{itemize}

\item  Entropic optimal transport problem, for a Lipschitz cost $h$:
\begin{align*}
EOT_{\eps}(\mu,\nu) &= \inf_{ \pi \in \Pi(\mu,\nu)} \left \{ \int_{\R^{2d}} h(x-y) \, d \pi + \eps Ent ( \pi | \mu \otimes \nu)  \right\} \\
& = \sup_{\phi,\psi } \left\{ \int \phi \, d \mu + \int \psi \, d \nu \; : \; \int e^{\frac{\phi(x)+\psi(y)-h(x-y)}{\eps} } \, d \mu \otimes  \nu =1 \right\}.
\end{align*}
Assuming that $G$ is $\lambda$-convex we expect the following inequality to hold
$$ \int \nabla G (\nabla \phi)  \cdot \nabla \mu +  \nabla G (\nabla \psi) \cdot \nabla \nu \,dx \geq  \lambda \eps\left( \int \frac { |\nabla \mu|^2}{\mu} + \int \frac { |\nabla \nu|^2}{\nu} \right);$$
notice that the constraint in the dual is not translation invariant: in fact it is useful in this case to pass to the potentials $u(x)=\phi(x) +\eps \ln(\mu (x))$ and $v(x)=\psi(y) + \eps \ln(\nu(y))$, for which we can obtain the \emph{five gradients inequality} and then use the convexity of $G$ to come back to $\phi$ and $\psi$.
\item Hellinger-Kantorovich distance:
\begin{align*}
HK(\mu, \nu)^2 = \sup_{u,v} \Big\{ \int u \, d \mu  + &\int v \, d \nu \; : u(x)<1, v(y)<1,  \\ & . (1-u(x))(1-v(y))\geq \cos^2(\min\{|y-x|,\pi/2\}) \Big\}.
\end{align*}
Then, the following inequality is expected to hold for every convex function $G$:
$$ \int \nabla G (\nabla \phi)  \cdot \nabla \mu +  \nabla G (\nabla \psi) \cdot \nabla \nu \,dx \geq  0$$  
\end{itemize}

\paragraph{Structure of the paper}
In Section 2 we discuss the setting of Lie groups and prove Theorem \ref{thm:5gi Lie Group}; Section 3 encompasses some known results about the variation of the arc-length and optimal transport theory on Riemannian manifolds which will be exploited in Section 4 to show that suitable second variations of the arc-length are bounded from above by the Ricci curvature. In the first part of Section 5 we make use of the construction built in Section 4 to recover a weighted pointwise four gradients inequality for smooth enough initial measures, which eventually implies the nonlinear five gradients inequality. Finally, the second part of Section 5 is devoted to the proof of Theorem \ref{thm:main}.

\paragraph{Acknowledgements}
This work started while S.D.M. was a postdoc in Orsay under the funding of ANR project ISOTACE: he would like to thank Filippo Santambrogio, Young-Heon Kim and François-Xavier Vialard for useful suggestions and for sparking the interest in the geometric aspect of the \emph{five gradients inequality}.
We would like to thank also Guido De Philippis and Alessio Figalli for helpful discussions related to the topic of this paper. We are grateful to Nicola Gigli for technical discussions and suggestions about this paper.
We are grateful to the referee for useful comments on the
manuscript.
 This work was written within the activities of GNFM and GNAMPA groups of INdAM.

\paragraph{Fundings}
 S.M. acknowledges the support of the GNFM-INdAM and 
the INFN national project “Bell”. E.R. acknowledges the support of the national INdAM project n.E53C22001930001 ``MMEAN-FIELDSS''.  S. D. M. is a member of GNAMPA (INdAM) and acknowledges the support of the AFOSR project FA8655-22-1-7034 and PRIN project no. 202244A7YL "Gradient Flows and Non-Smooth Geometric Structures with Applications to Optimization and Machine Learning". S.D.M. and S.M. were also supported in part by MIUR Excellence Department Project 2023-2027 awarded to the Dipartimento di Matematica of the University of Genova, CUP D33C23001110001.

\paragraph{Notation and conventions in Riemannian geometry}
\begin{itemize}
\item[-] $\M:=(\M,g)$ denotes a compact, connected, orientable, smooth Riemannian manifold.
\item[-] $\Gamma(E)$ denotes space of smooth sections of a principal bundle $E$ over $\M$.
\item[-]  $T\M$ and $F_O\M$ denote respectively the tangent and the orthonormal tangent frame bundle.
 \item[-] We denote the Levi-Civita connection with $\nabla$ and the Christoffel symbols with $\Gamma_{ij}^k$.
 \item[-] $\Pi_{b} V(a)$ denotes the parallel transport of the element $V(a)$ of the vector filed $V$ along the geodesic connecting $a$ and $b$.
 \item[-] We denote the \textit{Riemann curvature tensor} by $R:T_p\M\times T_p\M\times T_p\M\to T_p\M$ and we recall that,
  for any $v,w,z\in T_p\M$ it holds
 $$R(v,w)z:= [\nabla_v,\nabla_w] z - \nabla_{[v,w]} z \,.$$
 \item[-] We denote the \textit{Ricci curvature} by $\ric:T_p\M\times T_p\M\to \R$ and the \textit{Ricci tensor} by $\Ric_p:T_p\M\to T_p\M$ and we recall that, for any $p\in\M$ and for any $v,w\in T_p\M$ it holds
 $$\ric(v,w) =g(\Ric_p(v),w) = \sum_i g(R(v,e_i)e_i,w)\,.$$
 \item[-] We denote the scalar curvature by $\scal (p): \M\to\R$ and we recall that, for any $p\in\M$ and for any orthonormal basis $\{e_i\}\subset T_p\M$, it holds
 $$\scal(p)=\sum_i  \ric(e_i,e_i)= 	\sum_i g(\Ric_p(e_i),e_i)\,.$$
\end{itemize}

\label{sec:grad ineq}
\subsection{Setup} \label{setup}
Throughout this paper, we shall adopt the following setup:
\begin{itemize}
\item[(man)] $(\M,g)$ is a smooth compact Riemannian manifold; 
we denote by $K$ the bound from below of the  Ricci curvature, and by $\tilde{K}$ the constant that uniformly bound all the sectional curvatures;
\item[(meas)] $\mu,\nu$ are probability measures on $\M$ absolutely continuous with respect to the volume form; we will use the same letters, with a slight abuse of notation, to denote their densities with respect to $\vol$ (which denotes the volume form induced by the metric $g$). 
\item[(cost)] the cost function $c: \M\times \M \to \R$ is defined by 
$c(x,y):= h(d(x,y))$, where $d$ denotes the Riemannian distance and $h \in C^1([0,\infty))$ is some nonnegative strictly convex increasing function such that $h(0)=0$. Equivalently $c(x,y)= \int_0^{d(x,y)} \lambda(\tau)d\tau$ for some strictly increasing continuous function $\lambda$.

\item[(cvx)]$\ell \in C^1([0,\infty))$ is an increasing convex function such that $\ell(0)=0$. 

\end{itemize}
Moreover, in order to prove the results in full generality, we will argue approximating with more regular objects, for which we use the following stronger assumptions:

\begin{itemize}

\item[(cost2reg)] the cost function $c$ satisfies (cost) where, additionally, $h(t)=f(t^2)$ for some $f \in C^2([0,+\infty))$.

\item[(cvxreg)]$\ell$ satisfies (cvx) and, additionally, there exists $g \in C^2([0,+\infty))$ such that $\ell(t)=g(t^2)$ and $g \equiv 0$ in a neighborhood of $0$.
\end{itemize}

\section{Five gradients inequality on locally compact Lie groups}
Let  $\G$ be a locally compact connected Lie group of dimension $n$ and denote with $l_g:\G\to\G$ (resp. $r_g:\G\to\G$) the left action (resp. the right action) defined by $l_g(x):=gx$ (resp. $r_g(x)=xg^{-1}$). 
As usual, we denote with $dl_g$ and $dr_g$ the lifts of these actions to the tangent bundle. Let us consider the Lie algebra $\mathfrak{g}$, the set of the right-invariant vector fields $X$ such that for any $g\in\G$ it holds $dr_g\circ X = X \circ r_g$, endowed with the product given by the Lie bracket.

Let now $\{X_i\}_{i=1}^d$ be an orthogonal set of right-invariant vector fields 
and denote with $\Phi^i_t$ be the flow of $X_i$, i.e. the map $\Phi: \R\times\G\to \G$ defined for each $t\in\R$ by sending  $x \in \G$ to the point obtained by following for time $t$ the integral curve starting at $x$ defined by $$\frac d{dt} \Phi^i_t(x)= X_i ( \Phi^i_t(x) )\,.$$
Notice that we allow for $d<n$, which is for example the setting of nilpotent Carnot groups.
Since the vector fields $X_i$ are chosen to be right-invariant, it is easy to see that 
$$\Phi^i_t(x)= \Phi^i_t(1_\G) \cdot x = l_{\Phi^i_t(1_\G) }(x)\,,$$ where $1_\G$ is the unit of $\G$.
By denoting with  $dx$ the right Haar measure of $\G$, we get 
\begin{equation}\label{eq:inv mu}
(\Phi^i_t)_{\sharp} dx = dx 
\end{equation} and by uniqueness of the flow  
$$\Phi^i_t(1_\G)=(\Phi^i_{-t}(1_\G))^{-1}\,.$$
This is because the right Haar measure is the unique (up to a positive multiplicative constant) countably additive, nontrivial measure $d\mu$  on the Borel subsets of $\G$ satisfying the following properties:
(1)
The measure $dx$  is right-translation-invariant; 
(2) 
The measure $dx$  is finite on every compact set;
(3) The measure $dx$  is outer regular on Borel sets $S\subseteq \G$;
(4)  The measure $dx$  is inner regular on open sets $U\subseteq \G$.
\begin{rmk}
Notice that, on a generic (smooth) Riemannian manifold $(\M,g)$, Equation~\eqref{eq:inv mu} can be only satisfied for divergence-free vector fields. However, not every manifold admits an orthogonal set of divergence-free vector fields, see e.g.~\cite{Divfree}.
\end{rmk}

A natural assumption to have on the vector fields $X_1, \ldots, X_d$ is the H\"{o}rmander condition: let $\Delta=\langle X_1, \ldots, X_d \rangle$ the distribution generated by $X_1, \ldots, X_d$ as a subspace of $\mathfrak{g}$, we say $\Delta$ satisfies the H\"{o}rmander condition if it generates the whole $\mathfrak{g}$ as an algebra. In this case, we will say that $(\G, \Delta)$ is a Lie group with its associated right invariant subRiemannian structure $\Delta$.

Then, an important concept is that of Horizontal regularity. We can look for rectifiable \emph{horizontal} curves $\gamma:[0,1] \to \G$, that is, curves such that $\dot{\gamma} \in \Delta=\langle X_1, \ldots, X_d \rangle$, and we accordingly define a distance, called Carnot-Caratheodory distance
$$ d_{CC}(x,y)= \min \left\{ \int_0^1\|\dot{\gamma}(t)\| \, dt \; : \;  \gamma \text{ horizontal curve such that } \gamma(0)=x, \gamma(1)=y \right\}$$
By the left invariance of the construction, it is obvious that, for every $h \in \G$, we have $d_{CC}(x, y) = d_{CC}(hx, hy)= d_{CC} ( y^{-1}x , 1_{\G} )$ and crucially, thanks to the H\"{o}rmander condition we have that $d_{CC}(x,y) < +\infty$, that is every two points can be connected by a finite length horizontal curve. Next, we say that $f: \G \to \R$ is horizontally Lipschitz (or H-Lipschitz) if it is Lipschitz with respect to the distance $d_{CC}$. 

We recall that thanks to Pansu differentiability theorem, every H-Lipschitz function is differentiable $dx$-a.e. along the horizontal directions; we identify by $\nabla_H f$ its gradient.

The transport costs we will be using are left invariant Lipschitz function, that is they are of the following form: $c(x,y)=h(y^{-1}x)$ for some $H$-Lipschitz $h: \G \to [0,\infty)$. In the sequel we will need a lemma about the strong $L^1$ convergence of differential quotients to directional derivative for horizontal Sobolev functions,  which can be defined as in the Euclidean case using the integration by parts formula that holds also in  $\G$: we say that a function $f \in L^1(\G)$ is a Sobolev function if for every $i=1, \ldots, d$ we have that there exists $g_i \in L^1(\G)$ such that 
\begin{equation}\label{eqn:intbypartsG} \int_{\G}  g_i \cdot \psi\, dx = - \int_{\G} f \cdot X_i (\psi) \, dx  \qquad \forall \psi \in C_c^{\infty}(\G).\end{equation}
In this case we say that $f \in W_H^{1,1}(\G)$ and we define $\nabla_H f = ( g_1, \ldots, g_d)$.

 \begin{lemma}\label{lem:convG} Let $f \in W^{1,1}_H(\G)$ and let  $\Phi_t$ be the flow of an horizontal right-invariant vector field $X \in \mathfrak{X}(\G)$. Then we have that
$$ \frac {f(\Phi_t(x))- f(x)}t \to Xf (x) \qquad \text{ in }L^1(\G)$$
\end{lemma}

\begin{proof} For $dx$-a.e. $x \in \G$ we have that $t \mapsto f(\Phi_t (x))$ is an absolutely continuous function whose derivative is $Xf ( \Phi_t(x))$. In particular for almost every $x$ we have
\begin{equation}\label{eqn:PtX} \frac{f(\Phi_t(x))- f(x)}t = \frac 1t \int_0^t Xf (\Phi_s(x))\, ds = P_t ( Xf) (x), \end{equation}
where $P_t : L^1(\G) \to L^1(\G) $ is defined as $P_t (g) (x):= \frac 1t \int_0^t g(\Phi_s(x)) \, ds $.  This equality can be proven also by duality and linearity starting from~\eqref{eqn:intbypartsG}.

 Using Fubini, Jensen  and the invariance of $dx$ under translation, it is easy to see that for every $p \geq 1$ we have $\|P_t g \|_p \leq \|g \|_p$. Moreover if $g$ is H-Lipschitz we also have $|P_t g(x) - g(x)| \leq {\rm Lip}(g) t/2$. By usual triangular inequality and approximation with compactly supported Lipschitz functions we obtain that $\|P_t g  -  g \|_1 \to 0$ for every $g \in L^1(\G)$.

Now, using this observation and \eqref{eqn:PtX} we deduce

$$ \frac {f(\Phi_t(x))- f(x)}t - Xf(x) = P_t(Xf)(x) - Xf(x) \to 0 \qquad \text{ in }L^1(\G)$$
which concludes our proof.
\end{proof}

We are now ready to prove the five gradients inequality on locally compact Lie groups.

\begin{proof}[Proof of Theorem~\ref{thm:5gi Lie Group}]
First of all the existence of $d$-Lipschitz Kantorovich potentials $\phi, \psi$ is given by Theorem \ref{struct}, in particular \eqref{duality}, which is true in every Polish metric space as long as $c$ is Lipschitz.

 Let  now $\Phi_t$ be the flow of the vector field $X$. Let us then define
 \begin{subequations}
\begin{gather}
  \tilde{\phi}_t := \phi +  f'  \left ( \frac { \phi_t - \phi }{t} \right) t  \qquad \text{with } \quad \phi_t (x):= \phi (\Phi_t(1_\G)\cdot x),\label{phi}
\\ 
 \tilde{\psi}_t  := \psi  +  f' \left ( \frac { \psi_t - \psi } {t} \right ) t  \qquad \text{with } \quad \psi_t (y):= \psi (\Phi_t(1_\G)\cdot y). \label{psi} 
 \end{gather}
\end{subequations}
Let us moreover consider the set
$$ S_{\phi,\psi}:= \{ x,y \in \G \,|\, \phi (x) + \psi(y) =  h(y^{-1} x)  \}. $$
A straightforward computation shows that
$$\phi_t (x) + \psi_t (y ) = \phi (\Phi_t(1_\G)\cdot x) + \psi (\Phi_t(1_\G)\cdot y) \leq h (y^{-1} x).$$ 
Furthermore, 
 for any $t>0$ and $x,y \in S_{\phi,\psi}$ it holds
\begin{equation}\label{eq:step}
\begin{aligned}
 \frac{\phi_t(x) -\phi(x)}t + \frac{\psi_t(y)-\psi(y)}t &= \frac{ \phi_t (x) + \psi_t (y) - \phi(x) - \psi (y) }{t}   \\
& =  \frac{ \phi_t (x) + \psi_t (y) - h(y^{-1} x) }{t}  \leq 0.
\end{aligned}
\end{equation}
Since $f$ is an even convex function, by the monotonicity of the derivative, for any $s\leq -t$  we have $f'(s) \leq f'(-t)$. Furthermore, since $f$ is even, its first derivative is odd, i.e. $f'(-t)=-f'(t)$, which allow us to conclude that for any $t+s \leq 0$ it holds
\begin{equation}\label{eqn:mono} 
f'(t) + f'(s) \leq 0. 
\end{equation}
Combining Inequality~\eqref{eq:step} with Equations~\eqref{phi} and~\eqref{psi}, and applying Inequality~\eqref{eqn:mono}, we obtain, for any $x,y \in S_{\phi,\psi}$,
$$\tilde{\phi}_t(x) + \tilde{\psi}_t(y) \leq h (y^{-1} x).$$ 
Finally, consider an optimal plan  $\gamma$ for $\mu$ and $\nu$. We know that it is concentrated on $S_{\phi, \psi}$ and its marginals are $\mu$ and $\nu$. Then we find that 
\begin{align*}
\int_{\G} \tilde{\phi}_t \, d\mu + \int_{\G} \tilde{\psi}_t d \nu &=  \int_{\G \times \G} ( \tilde{\phi}_t+  \tilde{\psi}_t ) \, d \gamma  =  \int_{S_{\phi, \psi} } ( \tilde{\phi}_t+  \tilde{\psi}_t ) \, d \gamma \\&  \leq \int_{S_{\phi, \psi} } ( \phi + \psi ) \, d \gamma  = \int_{\G} \phi \, d\mu + \int_{\G} \psi d \nu \,.
\end{align*}
In particular we can write the same inequality for $-t$ and then add them up to obtain
$$ \int_{\G} (\tilde{\phi}_t +\tilde{\phi}_{-t} - 2 \phi ) \, d \mu + \int_{\G} (\tilde{\psi}_t +\tilde{\psi}_{-t} - 2 \psi ) \, d \nu \leq 0. $$
Now we compute:
$$\tilde{\phi}_t +\tilde{\phi}_{-t} - 2 \phi  = f' \left( \frac{\phi_t - \phi}t \right) t + f' \left ( \frac { \phi_{-t} -\phi } t \right) t. $$
Letting $H(x)= f' \left( \frac{\phi_t(x) - \phi(x)}t \right)$, $\overline g_t:=\Phi_t(1_\G)$ and $\overline g_{-t}:=\Phi_{-t}(1_\G)=\overline g_t^{-1}$, we notice that 
\begin{align*}
 f' \left (\frac { \phi_{-t}(x) -\phi(x) } t \right)  &=f' \left( \frac { \phi(\overline g_{-t} \cdot x) -\phi(x) } t  \right) \\ 
 &=f' \left( \frac { \phi(\overline g_{t}^{-1} \cdot g) -\phi(\overline g_t \cdot \overline g_t^{-1} \cdot x) }t \right) = - H ( \overline g_t^{-1} \cdot x )\,. 
\end{align*}
In particular, using the change of variable $g \mapsto \overline g_t^{-1} \cdot g$, which leaves the Haar measure $dx$ invariant, we get 
\begin{align*}
  \int_{\G} (\tilde{\phi}_t +\tilde{\phi}_{-t} - 2 \phi ) \, d \mu &= \int_{\G}  t (H(x) - H(\overline g_t^{-1} \cdot x)) d \mu \\
 & = t\int_{\G} H(x) \mu(x) \, dx -  t\int_{\G} H(\overline g_t^{-1} \cdot x) \mu(\overline g_t \cdot (\overline g_t^{-1} \cdot x)) \, dx \\
 & = t\int_{\G} H(x) \mu(x) \, dx -  t\int_{\G} H(x) \mu(\overline g_t \cdot x) \, dx \\
 & = t\int_{\G}   \,  f' \left( \frac{\phi_t - \phi}t \right) (\mu - \mu_t ) \, dx;
 \end{align*}
we can do the same for the term with $\tilde{\psi}_t$ and then dividing by $t^2$ and then letting $t\to 0$ we can conclude using dominated convergence and Lemma~\ref{lem:convG}.

Let us now consider the horizontal tangent space $H$ which is isomorphic to $\R^d$ with some metric $g$, which we can assume to be the usual scalar product, up to a change of coordinates. For any $v \in H$ we can consider the relative right-invariant vector field $X_v$. We can now average \eqref{eqn:5gif} for $f=f_v$ to get
\begin{equation}\label{eqn:Sd1}  -\int_{S^{d-1}} \int_{\G} \Big( f_v' ( X_v \phi   )  X_v \mu  + f_v' ( X_v \psi  ) X_v \nu \Big) \, dx \, dv \leq 0.\end{equation}
Notice that $X_v \phi =  v \cdot  \nabla_H \phi$ and similarly the other terms. Denoting $H_v(w)=f_v(  v \cdot w)$ we have that $\nabla H_v (w)= f'_v( v \cdot  w ) v$, in particular we have
$$  f_v' ( X_v \phi   )  X_v \mu = f_v'(  v \cdot  \nabla_H \phi   )   v \cdot \nabla_H \mu  = \nabla H_v ( \nabla_H \phi ) \cdot \nabla_H \mu,$$
and similarly for the other term in \eqref{eqn:Sd1}. Using now Fubini and the linearity of the gradient we see that \eqref{eqn:Sd1} is equivalent to 
$$  - \int_{\G} \Big( \nabla G( \nabla_H \phi   ) \cdot \nabla_H \mu  + \nabla G( \nabla_H \psi   ) \cdot \nabla_H \nu \Big)\,  dx \, \leq 0.$$
Notice that if $f_v$ is independent of $v$, then $G$ is rotation invariant and so $G(w)=g(\|w\|)$ for some convex $g$. When $f_v(t)=|t|^p$ we get precisely $g(t)=c_p |t|^p$ for some $c_p >0$, which let us conclude.
\end{proof}

\section{Preliminary results in Riemannian geometry}

In this section, we collect some facts about optimal transport and Riemannian geometry that we will need in the sequel. We refer  the reader to \cite{Villani,MC} for more details on the optimal transport theory (in the particular case of Riemannian manifolds) and to~\cite{Baer,docarmo} for what concerns the variation of the arc length on Riemannian manifolds.

\subsection{Optimal transport on compact Riemannian manifolds with general costs}

We begin by recasting the definition of a semi-concave function on a Riemannian manifold.

\begin{definition}
Let $U$ be an open set of a smooth Riemanian manifold $(\M,g)$ and let $\omega : (0,\infty) \to(0,\infty)$ be continuous, such that $\omega(r) = o(r)$ as $r \to 0$. A function $\phi : U\subset \M \to \R \cup \{+\infty\}$ is said to be \emph{semi-concave} with modulus $\omega$ if, for any constant-speed geodesic path $\gamma(t)$, whose image is included in $U$, it holds
$$(1-t)\phi(\gamma(0))+t \phi(\gamma(1)) - \phi(\gamma(t))\leq t(1-t)\omega (d(\gamma(0),\gamma(1)))$$
where $d$ is the Riemannian distance defined by Equation~\eqref{eq:Riem dist}. In particular, we will say that a function is $\Lambda$-concave if it is semi-concave with $\omega(t)=\Lambda t^2$.
\end{definition}

\begin{rmk}\label{rmk:semiconcavity}
Notice that the cost function $d^2$ fails to be semiconvex at the cut locus (cf.~\cite[Proposition 2.5]{Cor}), however it is semiconcave (cf.~\cite[Corollary 3.3]{Cor}), which is equivalent to requiring $\frac {d^2}{dt^2} d^2(\gamma_t,y) \leq C$ in the viscosity sense for every geodesic $\gamma_t$. Moreover under (cost2reg) $c$ is also semiconcave; in fact, letting $\psi(t)= d^2(\gamma_t,y)$, we get
$$ \frac{d^2}{dt^2} c(\gamma_t,y) =   \frac{d^2}{dt^2} f(\psi(t)) = f''( \psi (t)) \psi'(t)^2 + f'(\psi(t)) \psi''(t) \leq \tilde{C},$$
where we used that $\psi, \psi'$ are bounded in terms of the diameter of $\M$, and that $f'$ and $f''$ are bounded on bounded sets.
\end{rmk}

Let now $\Prob(\M)$  be the set of probability measures in $\M$, a smooth and compact Riemannian manifold.
\begin{thm}\label{struct}
Let  $\mu,\nu \in \Prob(\M)$ be two probabilities on $\M$ which are absolutely continuous with respect to $\voll$: with a slight abuse of notation we will sometimes identify $\mu$ and $\nu$ with their densities with respect to $\voll$. Let us consider a cost function $c=h \circ d$ that satisfies (cost) for some strictly increasing continuous function $\lambda:[0,+\infty) \to [0,+\infty)$. Then the following hold:
\begin{enumerate}
\item[(i)] The problem
\begin{equation}\label{kantorovich}
\Ch(\mu,\nu):=
\min\left\{\int_{\M\times\M} c(x,y) \, d \gamma\;:\;\gamma\in\Pi(\mu ,\nu)\right\},
\end{equation}
where
$\Pi(\mu ,\nu)$ is the set of  {\em transport plans}, i.e.
 $$\Pi(\mu,\nu):=\{\gamma\in\Prob(\M\times\M):\,(\pi^x)_{\#}\gamma=\mu ,\,(\pi^y)_{\#}\gamma=\nu\},$$
  has a unique solution, which is of the form $\gamma_{\T}:=(id,\T)_\#\mu$, and $\T:\M\to\M$ is a solution of the problem 
\begin{equation}\label{transomega}
\min_{\T_\# \mu=\nu } \int_{\M}h (d(x,\T(x))) \, d \mu (x) \,.
\end{equation}
\item[(ii)] The map $ \T: \{\mu>0\}\to \{\nu>0\}\) is $\voll$-a.e. invertible and its inverse $\S:=\T^{-1}$ is a solution of the problem 
\begin{equation}\label{transomegaS}
\min_{\S_\# \nu=\mu} \int_{\M} h(d(\S(y),y)) \, d \nu(y). 
\end{equation}
\item[(iii)] We have 
\begin{equation}\label{duality}
\Ch (\mu,\nu)=\max_{\phi,\psi\, \in Lip(\M)}\left\{ \int_{\M}\phi(x)\, d \mu(x) +\int_{\M}\psi(y) \, d \nu(y) \ :\  \phi(x)+\psi(y)\le c(x,y),\ \forall x, y\in\M\right\}.
\end{equation}
\item[(iv)]
If $\phi,\psi$ are optimal in \eqref{duality}, they are clearly Lipschitz and differentiable almost everywhere; moreover we have:
\begin{itemize}
\item $\T(x)=\exp_x( \lambda^{-1}(-\nabla\varphi(x)))$  and  $\S(y)=\exp_y(\lambda^{-1}(-\nabla\psi(y)))$, almost everywhere; in particular, the gradients of the optimal functions are uniquely determined (even in case of non-uniqueness of $\phi$ and $\psi$) a.e. on $\{\mu>0\}$ and $\{\nu>0\}$, respectively;
\item if $c=h \circ d$ satisfies (cost2reg) the functions $\phi$ and $\psi$ are $\Lambda$-concave for some $\Lambda \in \R$;
\item $\phi(x)=\min_{y\in\M}\left\{h(d(x,y))-\psi(y)\right\}\qquad\hbox{and}\qquad \psi(y)=\min_{x\in\M}\left\{h(d(x,y))-\phi(x)\right\};$
\item if we denote by $\chi^c$ the  $c-$transform of a function $\chi:\M\to\R$ defined through $\chi^c(y)=\inf_{x\in\M} \{c(x,y)-\chi(x)\}$, then the maximal value in \eqref{duality} is also equal to 
\begin{equation}\label{cconc}
\max_{\phi \, \in C^0(\M)} \left\{ \int_{\M}\phi(x)d \mu(x) +\int_\M \phi^c(y)d \nu(y) \right\}
\end{equation}
and the optimal $\phi$ is the same as above, and is such that $\phi=(\phi^c)^c$ a.e. on $\{\mu >0\}$.
\end{itemize}
\item[(vi)] If $\nu \in \Prob(\M)$ is given, the functional $\mathcal{F}:\Prob(\M)\to\R$ defined through
$$\mathcal{F}(\mu)=\Ch(\mu,\nu)=\max_{\phi \in C^0 (\M)} \left\{ \int_{\M}\phi(x) \, d \mu(x)+\int_\M \phi^c(y)\, d \nu(y) \right\}$$
is convex. Moreover, if $\{\nu>0\}$ is a connected open set we can choose a particular potential $\hat\phi$, defined as
$$\hat\phi(x)=\inf\left\{h(d(x,y))-\psi(y)\,:\,y\in \spt (\nu)\right\},$$
where $\psi$ is the unique (up to additive constants) optimal function $ \psi$ in \eqref{duality} 
(i.e. $\hat\phi$ is the $c-$transform of $\psi$ computed on $\M\times\spt(\nu)$). With this choice, if $\chi=\tilde\ro-\ro$ is the difference between two probability measures, then we have
$$\lim_{\ve\to 0} \frac{\mathcal{F}(\varrho+\ve\chi)-\mathcal{F}(\varrho)}{\ve}=\int_\M \hat\phi\,\dd\chi.$$ As a consequence, $\hat \phi$ is the first variation of $\mathcal{F}$ and from point (v) we deduce that $\hat \phi$ coincides with $\phi$ of the optimal couple $(\phi,\psi)$ in \eqref{duality}. 
\end{enumerate} 
\end{thm}

The only non-standard point is the last one. For more details we refer to \cite[Section 7.2]{OTAM}, (see also~\cite{Buttazzo} for a sketch of the proof.) Uniqueness of $ \psi$ on $\spt(\nu)$ is obtained from the uniqueness of its gradient and the connectedness of  $\{\nu >0\}$. 
\smallskip

We can now state (and prove) the following result, which is a direct consequence of~\cite{DF,Cor,MC}.
\begin{thm}\label{thm:DF}
Let $\M$ be a smooth Riemannian manifold, and let $\mu,\nu : \M \to (0,\infty)$  be two continuous probability densities, locally bounded away from zero and infinity on $\M$, and let us consider a cost function $c: \M \times \M \to [0, +\infty)$ satisfying (cost2reg). Let $\T: \M \to \M$
denote the optimal transport map for the cost $c$  sending $\mu$ onto $\nu$. Then there exist two closed
sets $\Sigma_\mu$, $\Sigma_\nu \subset \M$ of measure zero such that $\T: \M \setminus \Sigma_\mu \to \M \setminus \Sigma_\nu$ is a homeomorphism of class $C_{loc}^{0,\beta}$ for any $\beta < 1$. In addition, if both $\mu$ and $\nu$ are of class $C^{k,\alpha}$ then $\T: \M \setminus \Sigma_\mu \to \M \setminus \Sigma_\nu$ is a
diffeomorphism of class $C_{loc}^{k+1,\alpha}$.
\end{thm}
\begin{proof}
As shown in \cite[Theorem 13]{MC}, given two probability densities $\mu$ and $\nu$ supported on $\M$, there exists a $c$-convex function $u: \M \to \R \cup \{\infty\}$ such
that $u$ is differentiable $\mu$-a.e., and $\T_u(p) = \exp_p (\lambda^{-1}(\nabla u(p)))$ (where $\lambda$ is that of (cost)) is the unique optimal transport map sending $\mu$ onto $\nu$. 
Furthermore, as noted in Remark~\ref{rmk:semiconcavity}, $c$ is semiconcave; in particular, the semiconcavity is inherited by the optimal Kantorovich potential $\phi$, and so $u=-\phi$ is semiconvex. By Alexandrov's Theorem, we get that $u$ is twice differentiable almost everywhere and, 
arguing as in \cite[Proposition 4.1]{Cor}, we conclude that  $\T_u (p)$ is not in the cut-locus of $p$. Since the cut-locus is closed and $c$ is smooth outside the cut-locus, there exists a set $X$ of full measure such
that, for every $p_0 \in X$, $u$ is twice differentiable at $p_0$ and there exists a neighborhood
$U_{p_0} \times V_{\T_u(p_0)} \subset \M \times \M$ of $(p_0, \T_u (p_0))$ such that $c \in C^\infty(U_{p_0} \times V_{T_u(p_0)})$. By taking a
local chart around $(p_0, \T_u (p_0))$ we reduce ourself to \cite[Theorem 1.3]{DF}. This shows that $\T_u$
is a local homeomorphism (resp. diffeomorphism) around almost every point. In particular, since $\T_u$ is invertible a.e., it can be shown that $\T_u$ is also a global homeomorphism
(resp. diffeomorphism) outside a closed singular set of measure zero. For more details we refer to~\cite{DF}.
\end{proof}

It follows that $\T$ has the form $\T(p):= \exp_p (\lambda^{-1}(\nabla u))$ for a proper $c$-convex function $u$ defined on $\M$ which is smooth on $\M \setminus \Sigma$ and $\Sigma$ is a closed set of zero volume;  notice that under the assumption (cost2reg) we also have that the vector field $\Xi=\lambda^{-1}(\nabla u)$, the generator of $\T$, is at least $C^2$ on the open set $\M \setminus \Sigma$. Thus, $\T$ is a diffeomorphism only on $\M \setminus \Sigma$. Moreover, there exist two functions $\phi,\,\psi : \M \to \R$, smooth on $\M \setminus \Sigma$, which are the Kantorovich potentials of the optimal transport problem. Since $h$ is bounded and Lipschitz, we can assume that $\phi,\, \psi, \, \nabla \phi$ and $\nabla \psi$ are uniformly bounded on $\M$.
Actually, $\phi = \psi^c$ and the optimality of the Kantorovich potentials implies 
\begin{equation}\label{da trasporto ottimo 1}
\int_\M \phi d\mu + \int_\M \psi d\nu = \sup\left\lbrace \int_\M \widetilde{\phi} d\mu + \int_\M \widetilde{\psi} d\nu \text{ among } (\widetilde{\phi}, \widetilde{\psi})\, :  \widetilde{\phi}(a) + \widetilde{\psi}(b) \leq h(d(a,b)) \text{ on } \M \times \M \right\rbrace 
\end{equation}
and
\begin{equation}\label{da trasporto ottimo 2}
\phi(q) + \psi(\T(q)) = h(d(q,\T(q))) \quad \text{ in   $ \M \setminus \Sigma$ }
\end{equation}

We conclude this section, by showing that the optimal transport map is stable.

\begin{prop}[Stability of optimal transport]\label{prop:stability} Let $\M$ be a compact manifold, $c_n,c: \M \times \M \to [0, +\infty)$ be $L$-Lipschitz uniformly bounded cost functions such that $c_n$ uniformly converges to $c$. Let $\mu_n, \nu_n $ be probability measures on $\M$, let $\gamma_n$ be  an optimal plan for $\mathcal{C}_{c_n}(\mu_n, \nu_n)$ and let $\phi_n, \psi_n$ be $c_n$-concave optimal Kantorovich potentials. 
 We further assume that
 $\phi_n, \psi_n$ are  $L$-Lipschitz and uniformly bounded.
Suppose that $\mu_n \rightharpoonup \mu$ and $\nu_n \rightharpoonup \nu$. Then, up to subsequences we have that $\phi_n$ uniformly converges to $\phi$, $\psi_n$ uniformly converges to  $\psi$ and $\gamma_n \rightharpoonup \gamma$ where $\gamma, \phi, \psi $ are respectively an optimal plan and optimal potentials for $\mathcal{C}_c(\mu, \nu)$.

If in addition $\mu_n, \mu \ll \vol$, and $c_n, c$ satisfy (cost), we also have $\nabla \phi_n \to \nabla \phi$ $\vol$-a.e.
\end{prop}

\begin{proof} We already know that we can assume $\phi_n, \psi_n$ to be $L$-Lipschitz and uniformly bounded. By Ascoli-Arzelà we can thus extract subsequences that converge uniformly to $\phi, \psi$; passing to the limit $\phi_n(x)+\psi_n(y) \leq c_n(x,y)$ we obtain $\phi(x)+\psi(y) \leq c(x,y)$, that is $\phi$ and $\psi$ are admissible potentials for $c$. Thanks to the weak compactness of $\mathcal{P}(\M \times \M)$ we can also assume that $\gamma_n \rightharpoonup \gamma$; it is easy to see that since $\gamma_n \in \Pi(\mu_n, \nu_n)$ we have $\gamma \in \Pi(\mu, \nu)$. If $c_n$ satisfy (cost) then, arguing similarly to \cite[Lemma 7]{MC}, in every point $x$ of differentiability of $\phi_n$, we have that $\phi_n(x) + \psi_n (y) = c_n(x,y)$ on the support of any optimal transport plan (in particular, on the support of $\gamma_n$), moreover 
\begin{equation}\label{eqn:equality} \phi_n(x)+\psi_n(\T_n(x))=c_n(x,\T_n(x)),\end{equation}
where $\T_n(x)=\exp_x (\lambda_n^{-1}(-\nabla \phi_n(x)))$.
We can now compute
\begin{align*}
\left| \int_{\M} \phi(x) \, d\mu + \int_{\M} \psi (y) \, d \nu - \int_{\M \times \M} c(x,y) \, d \gamma\right| &=\lim_{n \to \infty} \left|  \int_{\M \times \M}  (\phi + \psi -c) \, d \gamma_n \right|  \\
 & =  \lim_{n \to \infty}\left| \int_{\M \times \M} (\phi+ \psi-c - (\phi_n+\psi_n - c_n) ) \, d \gamma_n \right|  \\
 & \leq\lim_{n \to \infty}  \|\phi - \phi_n\|_{\infty} + \| \psi- \psi_n\|_{\infty} +\| c- c_n\|_{\infty} =0;
\end{align*}
we deduce that $\gamma, \phi, \psi$ are respectively an optimal plan and optimal potentials.

Let us consider $A$ the set of full $\vol$ measure where $\phi$ and $\phi_n$ for every $n \in \mathbb{N}$ are differentiable. Let us fix $x \in A$ and let $\bar{y}$ be a limit point for the sequence $\T_n(x)$; using the uniform continuity of $\phi_n, \psi_n, c_n$ and their uniform convergence, passing to the limit \eqref{eqn:equality} we obtain $\phi(x)+\psi(\bar{y})=c(x,\bar{y})$. Arguing again as \cite[Lemma 7]{MC} we get $\bar{y}=  \exp_x (\lambda^{-1}(-\nabla \phi(x)))=\T(x)$; since this holds true for every limit point we get $\T_n (x) \to \T(x)$ in $A$. Moreover, as in \cite[Lemma 7]{MC}, we get that $x' \mapsto d(x',\bar{y})$ is differentiable in $x$ as well and so there is a unique geodesic between $x$ and $\bar{y}$, which in turn implies $\nabla \phi_n (x) \to \nabla \phi(x)$.
\end{proof}

\subsection{Variations of the arc-length}\label{sec:variations}

Let $(\M,g)$ be a connected, oriented, compact, smooth Riemannian manifold of dimension $n$. On account of the Hopf-Rinow theorem,  $\M$ is geodesically complete, i.e. for every $p \in \M$, the exponential map $\exp_p:T_p\M\to\M$ is defined on the entire tangent space.  This assumption implies that for any two points $p$ and $q$ in $\M$, there exists a length minimizing geodesic 
\begin{align*}
\gamma: & [a,b]\to\M \qquad \text{ with } \qquad\gamma(a)=p \; \text{ and } \; \gamma(b)=q
\end{align*} 
and their Riemannian distance $d$ coincides with the arc-length of $\gamma$, namely
\begin{equation}\label{eq:Riem dist}
d(p,q)\equiv L(\gamma):=  \int_a^b \|\dot{\gamma}\|\, dt\,, 
\end{equation}
where $\dot{\gamma}$ denotes the tangent vector to $\gamma$ and $\|\dot{\gamma}\|:=g(\dot{\gamma},\dot{\gamma})^{\frac{1}{2}}$.  Since the arc-length of $\gamma$ is invariant under reparametrization, we parametrize $\gamma$ by arc-length, namely
$$ \gamma:[0,\ell] \to \M \qquad \text{ and } \qquad \|\dot{\gamma}\|=1\,$$
where $\ell = L(\gamma)$.
Consider a smooth orthonormal frame $\{e_i\}_{i=1,\dots,n}$ defined in an open neighborhood $\cU$ of $\gamma$.
Each vector field $e_i$ is the infinitesimal generator of a (local) flow which we denote as $\Phi_i$.
Then, for $\varepsilon >0$ small enough, we can define $n$-smooth variations of $\gamma$ via
\begin{equation*}\label{eq:variation}
   f_i : [0,\varepsilon]\times [0,\ell]\to \cU,\qquad (s,t)\mapsto f_i(s,t):= \Phi_i(s,\gamma(t))\,.
\end{equation*}
For every fixed $s \in [0,\varepsilon]$ we denote with $\gamma^s_i:= [0,\ell] \to \M $ the curve obtained as a variation of $\gamma$ whose endpoints are $f_i(s,0)$ and $f_i(s,\ell)$ and whose variational vector fields are exactly $e_i|_{\gamma}=:\xi_i$.
By denoting with $\dot{\gamma}_i^s$ the tangent vector to $\gamma_i^s$,  we have the following:
\begin{lemma}\label{lem:variationslength}
Let $\gamma_i^s$ be a variation of geodesic $\gamma$ connecting $p$ to $q$ defined as above. Then the first and second variations of the arc-length $L(\gamma_i^s)$ satisfy respectively
\begin{equation}\label{1st variation}
\frac{d}{ds}\Big|_{s=0}L(\gamma_i^s) =  g\big(\xi_i,\dot{\gamma} \big) \Big|_0^\ell
\end{equation}
\begin{align}\label{2nd variation}
 \frac{d^2}{ds^2}\Big|_{s=0}L(\gamma_i^s) \leq - \Big( g\big(\xi_i,\dot{\gamma} \big) \Big|_0^\ell\Big)^2 +  g\big( \nabla_{\xi_i} \xi_i, \dot{\gamma }\big)\Big|_0^\ell   +  \int_0^\ell \Big( \| \nabla_{\dot{\gamma}} \xi_i\|^2 - g\big(R(\xi_i,\dot{\gamma})  \dot{\gamma},\xi_i \big) \Big) dt  
\end{align} 
\end{lemma}
\begin{proof}
Parametrizing $\gamma$ by arc-length and calling $E$ the energy of $\gamma_i^s$, we have
\begin{align*}
\frac{d}{ds}\Big|_{s=0}L(\gamma_i^s) &= \frac{1}{2} \int_0^\ell \left. \left(\frac{1}{ g(\dot\gamma_i^s,\dot\gamma_i^s)^\frac{1}{2}} \frac{d}{ds} g(\dot{\gamma}_i^s,\dot{\gamma}_i^s) \, \right) \right|_{s=0} dt \\
&= \frac{1}{2} \int_0^\ell \left. \frac{d}{ds}\right|_{s=0} g(\dot{\gamma}_i^s,\dot{\gamma}_i^s) \, dt = \frac{d}{ds}\Big|_{s=0}E(\gamma_i^s)\\
\frac{d^2}{ds^2}\Big|_{s=0}L(\gamma_i^s) &= \frac{1}{2} \int_0^\ell \left. \frac{d}{ds} \left(\frac{1}{ g(\dot\gamma_i^s,\dot\gamma_i^s)^\frac{1}{2}} \frac{d}{ds} g(\dot{\gamma}_i^s,\dot{\gamma}_i^s) \, \right) \right|_{s=0} dt\\
&= -\frac{1}{4} \int_0^\ell \left.  \left( \frac{d}{ds} g(\dot{\gamma}_i^s,\dot{\gamma}_i^s) \, \right)^2 \right|_{s=0} dt + \frac{1}{2} \int_0^\ell \left. \frac{d^2}{ds^2}\right|_{s=0} g(\dot{\gamma}_i^s,\dot{\gamma}_i^s) \, dt \\
&= -\frac{1}{4} \int_0^\ell \left.  \left( \frac{d}{ds} g(\dot{\gamma}_i^s,\dot{\gamma}_i^s) \, \right)^2 \right|_{s=0} dt + \frac{d^2}{ds^2}\Big|_{s=0}E(\gamma_i^s)\\
&\leq - \left( \frac{d}{ds}\Big|_{s=0}E(\gamma_i^s)\right)^2 + \frac{d^2}{ds^2}\Big|_{s=0}E(\gamma_i^s).
\end{align*}
Clearly, since $\gamma$ is a geodesic, then $\dot{\gamma}$ is parallel transported along $\gamma$, i.e. $\nabla_{\dot{\gamma}} \dot{\gamma} = 0$. 
By using 
\cite[Theorem 2.6.5]{Baer} together with the computations in the proof of \cite[Theorem 5.3.2]{Baer}, we obtain the result.
\end{proof}

\begin{rmk}\label{variazioni + e -}
Let now $\gamma^{-s}_i $ be the variations of $\gamma$ obtained with the variational vector fields $-\xi_i$ and consider the Taylor expansion of $L(\gamma^{(\cdot)}_i)$ with respect to $s$, namely
\begin{equation*}\label{Taylor exp}
L(\gamma^s_i) =  L(\gamma) + \frac{d} {ds}\Big|_{s=0}L(\gamma^s_i) \, s + \frac{1}{2}\frac{d^2}{ds^2}\Big|_{s=0}L(\gamma^s_i)\, s^2 + o(s^2).
\end{equation*}
Using \begin{equation*}
\frac{d}{ds}\Big|_{s=0}L(\gamma_i^s) + \frac{d}{ds}\Big|_{s=0}L(\gamma_i^{-s}) = 0 \quad \mbox{ and } \quad \frac{d^2}{ds^2}\Big|_{s=0}L(\gamma_i^s) = \frac{d^2}{ds^2}\Big|_{s=0}L(\gamma_i^{-s})
\end{equation*}
we can immediately conclude that 
$$  L(\gamma^s_i) + L(\gamma^{-s}_i) = 2L(\gamma) +  \frac{d^2}{ds^2}\Big|_{s=0}L(\gamma^s_i) s^2  + o(s^2)\,. $$
Moreover, if $h$ is an increasing, non-negative $C^2$-function  function, then 
 \begin{equation}\label{Taylor exp con h}
 \begin{aligned}
 h(L(\gamma^s_i)) =\, & h(L(\gamma)) + s h'(L(\gamma)) \frac{d} {ds}\Big|_{s=0}L(\gamma^s_i)  + \frac{s^2}{2} h'(L(\gamma)) \frac{d^2}{ds^2}\Big|_{s=0}L(\gamma^s_i)  +\\
  &\hspace{5cm}+ \frac{s^2}{2} h''(L(\gamma))\left(\frac{d} {ds}\Big|_{s=0}L(\gamma^s_i)\right)^2   + o(s^2)
\end{aligned}
\end{equation}
 and, similarly as before, we get
\begin{equation*}\label{utile con h}
\begin{aligned}
h(L(\gamma^s_i)) + h(L(\gamma^{-s}_i)) = 2h(L(\gamma)) + s^2 & \left[h''(L(\gamma))\left(\frac{d} {ds}\Big|_{s=0}L(\gamma^s_i)\right)^2 + \right. \\
&  \qquad \left. +  h'(L(\gamma))\frac{d^2}{ds^2}\Big|_{s=0}L(\gamma^s_i) \right]+ o(s^2).
\end{aligned}
\end{equation*}
\end{rmk}

\section{Simultaneous local variations on good normal neighborhoods}\label{sec:simultaneous}

This section is dedicated to the construction in Theorem~\ref{lemmone} of simultaneous local variations of geodesics between a point $q \in \M$ and its image $\T(q) \in \M$. 
Such variations are designed to behave as much as possible as those obtained by rigid translation in the Eucledian setting. 

For the sake of clarity we will describe briefly the construction:

\begin{enumerate}

\item Given $p$ and $\T(p)$ we consider an orthonormal frame in $p$ and its corresponding one in $\T(p)$, constructed by parallel transport on the geodesic connecting $p$ and $\T(p)$. Then we can consider small enough neighborhoods $\cU_p$ and $\T(\cU_p)$ such that in each neighborhood we have good control of the orthonormal frames defined via the exponential map from $p$ and $\T(p)$  ($V_i$ and $W_i$ respectively): for the details of the properties see Definition~\ref{small normal neighborhood} of $\varepsilon$-\emph{small normal neighborhood}.

\begin{center}
\includegraphics[scale=0.35]{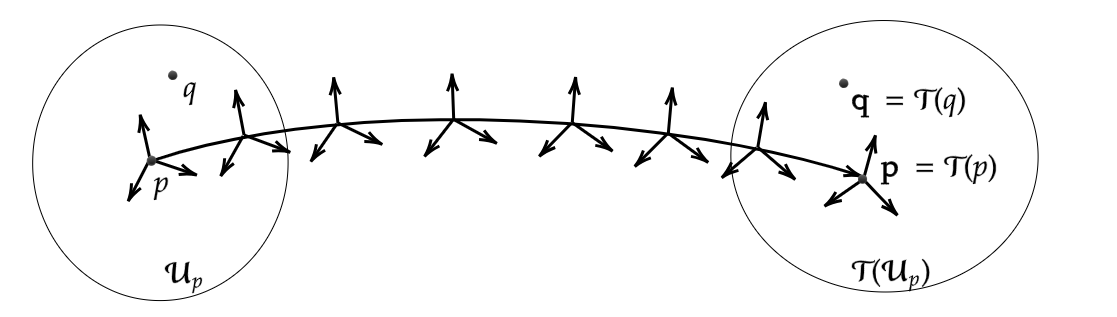}
\end{center}
\item For every $q \in \cU_p$ we can interpolate on the geodesic from $q$ to $\T(q)$ between the frames $V_i(q)$ and $W_i(\T(q))$. Notice that even though $\Pi_{\T(q)} V_i\neq W_i$ we still have $\Pi_{\T(q)} V_i \approx W_i$  (Lemma~\ref{lem:confrontoVW}) and so the interpolating frames $X_i$ are \emph{almost} parallel and \emph{almost} orthogonal. 

\begin{center}
\includegraphics[scale=0.35]{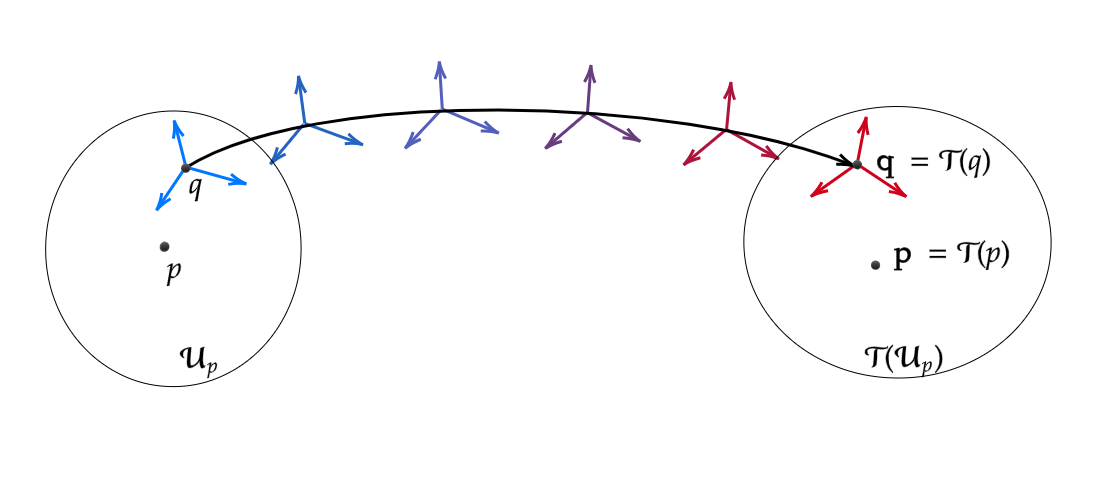}
\end{center}

\item Given $i \in \{ 1, \ldots n\}$ we can construct variations $\gamma_i^s$ of the geodesic $\gamma$ between $q$ and $\T(q)$ simply deforming the geodesic $\gamma$ through the flow of the vector field $X_i$, as shown below.
\begin{center}
\includegraphics[scale=0.35]{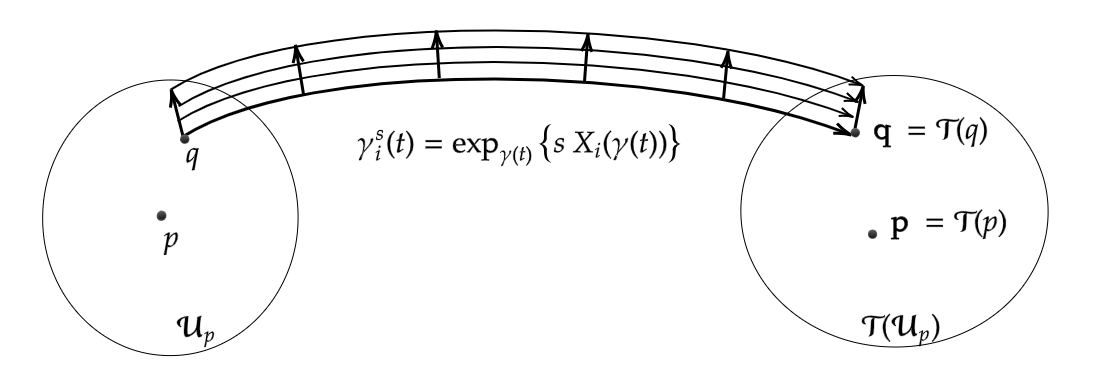}
\end{center}

\end{enumerate}

\subsection{Good normal neighborhood and $\varepsilon$-small normal neighborhood}

Now we introduce the notion of \emph{good normal neighborhood} of a point $p\in \M$. To this end, let $\cV_p\subseteq T_p\M$ be a neighborhood of $0\in T_p\M$ such that the exponential map is a diffeomorphism between $\cV_p$ and $\exp_p(\cV_p)=:\cN_p\subseteq \M$. We shall refer to $\cN_p$ as \emph{normal neighborhood} of $p$. 

Let $\cN_p$ be a normal neighborhood of $p$ and denote by $\Upsilon$ the isomorphism between $T_p\M$ and $\R^n$. Then we can use the exponential map $\exp_p: T_p\M\to\M $ to define a system of coordinates by
$$ (x_1,\dots,x_n):= (\exp\circ \Upsilon)^{-1} :  U_p\subset \R^n \to \cN_p\subset \M \,.$$ 
 We will refer to this system of coordinates as \emph{geodesic normal coordinates}. 
 This choice of coordinates induces a local coordinate frame $\{e_j\}:=\{\partial_{x_j}\}$ and the \emph{Christoffel symbols} of the Levi-Civita connection $\nabla$ are the $n^3$-functions defined by
$$ \nabla_{e_i}{e_j}=\sum_{k=1}^n \Gamma_{i,j}^k{e_k} \,. $$
 With the next proposition, we recall an important property of the geodesic normal coordinates. We refer to~\cite[Proposition 2.6.31]{Baer} for more details.
 \begin{prop}
 Let $p\in\cN_p\in\M$ and let $(x_1,\dots,x_n)$ be geodesic normal coordinates. Denote with $\Gamma_{ij}^k$ the Christoffel symbols  in geodesic normal coordinates. Then we have
$$ g(e_j,e_i)(p)=\delta_{ij} \qquad \Gamma_{ij}^k(p)=0\,, $$
where $\delta$ is the Kronecker delta. 
 \end{prop}
 
 \begin{cor}\label{cor:geod coor}
Let $\{e_j\}$ be a local frame for geodesic normal coordinates and define $\widetilde{e}_j\:=e_j/\|e_j\|$. 
Then for every $\varepsilon\in\R$, there exists a $\bar{\delta} = \bar{\delta}(\eps) \in \R$ such that for any $0 < \delta < \bar{\delta}$ the $n^3$-smooth function $ \tilde{\Gamma}_{ij}^k$ defined by
$\nabla_{\widetilde{e}_i}\widetilde{e}_j=\sum_{k=1}^n \widetilde{\Gamma}_{ij}^k \widetilde{e}_k$ satisfy
$$\| \tilde{\Gamma}_{ij}^k(q) \| \leq \varepsilon \qquad \text{ and }\qquad \| \nabla_{\widetilde{e}_l}\tilde{\Gamma}_{ij}^k(q) -  \nabla_{\widetilde{e}_l}\tilde{\Gamma}_{ij}^k(p) \| \leq \varepsilon\,. $$ 
for every $q$ in the geodesic ball $B_{\delta}(p)\subset \cN_p$ of radius $\delta$.
\end{cor} 
 \begin{proof} 
By a straightforward computation, we can see that
$$ \widetilde{\Gamma}_{ij}^k = \left( \frac{1}{\|e_j\|} \partial_{x_i}\|e_j\|^2+ \frac{1}{\|e_i\|\|e_j\|}\Gamma_{ij}^{k} \right) \|e_k\| \,,$$ 
where $\Gamma^k_{ij}$ are the Christoffel symbols in geodesic normal coordinates. Since for every fixed $i,j,k$, $\Gamma^k_{ij}$ are smooth functions, then we can conclude.
 \end{proof}

Let now $\T$ be a diffeomorphism between two open subsets of $\M$. Then, whenever $\cN_p$ is a normal neighborhood of $p$, also $\T(\cN_p)$ is a {normal neighborhood} of $\T(p)$.
Based on this observation we can finally introduce the definition of $\varepsilon$-\emph{small normal neighborhood} with respect to the diffeomorphism $\T$.
In what follows we  will make use of the following notation: given a map $\T: \M \to \M$ and a vector field $X: T\M \to T\M$, then $\Pi_{\T(p)} X(p)$ will denote the parallel transport of $X(p)$ along the geodesic between $p$ and $\T(p)$.

\begin{defn}\label{small normal neighborhood}
Let $\Omega, \Delta$ be open subsets of $\M$ and $\T \in\textit{Diff}(\Omega,\Delta)$ with infinitesimal generator $\Xi$, a section of $T\M$ such that $\T(p)=\exp_p( \Xi(p))$. Let then
$\varepsilon >0$ and $p \in \Omega$ be fixed. We say that $\cU_p\subset \cN_p$ is an \emph{$\varepsilon$-small normal neighborhood of $p$ with respect to $\T$} if the following holds true:
\begin{enumerate}[label=(\roman*)]
\item  $\cU_p$ is open and $\cU_p$ and $\T(\cU_p)$ are $\varepsilon$-separated, that is 
\begin{equation}\label{eqn:separation} d(p',q') \geq 4 \varepsilon \qquad \forall p' \in \cU_p , q' \in \T(\cU_q)\end{equation}
in particular we have also  $\cU_p \cap \T(\cU_p)= \emptyset$;
\item  $diam(\cU_p), diam(\T(\cU_p)) \leq \varepsilon$;
\item there exists $\delta < \bar{\delta}/2$, being $\bar{\delta}= \bar{\delta}(\varepsilon)$ the one provided by Corollary \ref{cor:geod coor}, such that the $\delta$-tubular neighborhood of $\cU_p$ is contained in the geodesic ball $B_{\bar{\delta}}(p)$ and the $\delta$-tubular neighborhood of $\T(\cU_p)$ is contained in $B_{\bar{\delta}}(\T(p))$;
\item it is possible to define two local frames for geodesic normal coordinates, $\{e_j\}$ on $\cU_p$ and $\{f_j\}$ on $\T(\cU_p)$, and the corresponding modifications $\widetilde{e}_j = e_j / \| e_j \|$ and $\widetilde{f}_j = f_j / \| f_j \|$, such that 
\begin{equation*}
\| \tilde{\Gamma}_{ij}^k(q) \| + \| \tilde{\Gamma}_{ij}^k(\T(q)) \| \leq \varepsilon \; \text{ and }\; \| \nabla_{\widetilde{e}_l}\tilde{\Gamma}_{ij}^k(q) -  \nabla_{\widetilde{e}_l}\tilde{\Gamma}_{ij}^k(p) \| + \| \nabla_{\widetilde{f}_l}\tilde{\Gamma}_{ij}^k(\T(q)) -  \nabla_{\widetilde{f}_l}\tilde{\Gamma}_{ij}^k(\T(p)) \|  \leq \varepsilon
\end{equation*}
for every $q \in \cU_p$, where, with a slight abuse of notation, we use the same notation $\tilde{\Gamma}_{ij}^k$ to denote the symbols on both the sets $\cU_p$ and $\T(\cU_p)$;
\item  $\widetilde{e}_1 (p)=  (\Xi/ \| \Xi \|)(p)$ and moreover $\widetilde{f}_i(\T(p))= \Pi_{\T(p)}\widetilde{e}_i(p)$ for all $i=1, \ldots, n$;
\item  for every $q \in \cU_p$ we have
\[ |(\Xi/ \| \Xi \|)(q) - (\Xi/ \| \Xi \|)(p)| < \varepsilon  \; \text{ and } \; |\Pi_{\T(q)}(\Xi/ \| \Xi \|)(q) - \Pi_{\T(p)} (\Xi/ \| \Xi \|)(p)| < \varepsilon \]
\item $\Xi$ is Lipschitz continuous in $\cU_p$.
\end{enumerate}
\end{defn}

\begin{lemma}\label{lemma: esiste sempre small normal neighborhood}  If $\T \in \textit{Diff}(\Omega,\Delta)$ is a diffeomorphism with infinitesimal generator $\Xi$, then every non-fixed point of $\T$ admits an $\varepsilon$-small normal neighborhood with respect to $\T$ inside $\Omega$  for any $\varepsilon >0$ small enough. In particular it is sufficient that $\varepsilon \leq C d(p, \T(p))$.
\end{lemma}
\begin{proof}
We aim to show that it is always possible to find a proper set of geodesic normal coordinates in $\cU_p$ and $\T(\cU_p)$ so that the requests of Definition \ref{small normal neighborhood} are satisfied. Let $p$ be a non-fixed point for $\T$ and $\bar{\delta}$ be the parameter provided by Corollary \ref{cor:geod coor}, then it is always possible to find $\delta < \bar{\delta}$ such that the geodesic balls $B_\delta(p)$ and $B_\delta(\T(p))$ are disjoint in $\M$, and in fact $\varepsilon$-separated. Up to further decrease $\delta$, conditions (i), (ii), (iii) of Definition \ref{small normal neighborhood} are clearly satisfied. Moreover, being the vector field $\Xi$ smooth on $\Omega$, also conditions (vi) and (vii) of Definition \ref{small normal neighborhood} must be true up to further decrease $\delta$.
Finally, observe that a rigid rotation inside $\R^n$ does not affect the notion of geodesic normal coordinates. As a consequence, it is always possible to find a system of geodesic normal coordinates inside $\cU_p$ so that $e_1(p)$ is parallel to $\Xi(p)$, thus $\widetilde{e}_1(p) = \Xi(p) / \| \Xi(p)\|$. In a similar way, we can find a system of geodesic normal coordinates inside $\T(\cU_p)$ so that $f_i(\T(p)) = \Pi_{\T(p)} e_i(p)$ (notice that orthogonality is preserved by parallel transport). Thanks to the estimates of Corollary \ref{cor:geod coor}, we conclude the proof. 
\end{proof}

\begin{rmk}\label{rmk:fixed}
Notice that a fixed point of $\T$ does not have a $\varepsilon$-small normal neighborhood. However, it is still possible to have properties (ii) to (vii) constructing a local normalized frame $\{\tilde{e}_i\}$ around $p$ on a normal neighborhood $\cN_p$ and then letting $\tilde{f}_i=\tilde{e}_i$: the conclusion follows taking a small enough neighborhood $\cU_p$.
\end{rmk}

In what follows, whenever $\cU_p$ is a $\varepsilon$-small normal neighborhood of $p$ with respect to $\T$, we will always denote $V_i(q) = \widetilde{e}_i(q)$ and $W_i(\T(q)) = \widetilde{f}_i(\T(q))$ for all $q \in \cU_p$ and $i \in \{1, \ldots, n\}$. We recall that $V_1(p) = (\Xi / \| \Xi \|)(p)$, while $W_i(\T(p))= \Pi_{\T(p)}V_i(p)$.

\begin{lemma}[Generalized Berger's Lemma]\label{rmk: per concludere il lemmone}
Let $(\M,g)$ be a $C^2$ compact Riemannian manifold and let $\tilde{K}$ be a uniform bound for the sectional curvatures. Then then for any $p \in \M$ and $u,v,w,z\in T_p \M$ one has 
\[  |g(R(u,v)w,z)|  \leq 7 \tilde{K} \|u\| \,\|v\| \,\|w\| \,\|z\|;  \]
By the arbitrariness of $z$ we also have $\|R(u,v)w\| \leq  7 \tilde{K} \|u\| \,\|v\| \,\|w\| $.
\end{lemma}

\begin{proof}
This can be understood as follows: Let $\{e_i\}$ an orthonormal basis for $T_p\M$ and we write 
$$ u = \|u\|\frac{u}{\|u\|}= \|u\| \sum_i \alpha_i e_i \qquad v = \|v\|\frac{v}{\|v\|}= \|v\| \sum_i \beta_i e_i $$  
$$w = \|w\|\frac{w}{\|w\|}= \|w\| \sum_i \rho_i e_i \qquad z = \|z\|\frac{z}{\|z\|}= \|z\| \sum_i \varrho_i e_i \,;$$
in particular, we have $\sum_i \alpha^2_i= \sum_i \beta_i^2=\sum_i \rho_i^2 =\sum_i\varrho_i^2 =1$; moreover, letting the orthonormal basis be obtained by Graham-Schmidt applied to a basis which starts with a properly selected linearly independent subset of $u,v,w,z$, we can assume $\alpha, \beta, \rho, \varrho$ have respectively at most one, two, three and four nonzero components. Then by linearity and using Berger's lemma (see e.g.\cite{Karcher})
\begin{align*}
|g(R(u,v)w,z) | &= \big|g( R(\|u\|\frac{u}{\|u\|},\|v\|\frac{v}{\|v\|})\|w\|\frac{w}{\|w\|},\|z\|\frac{z}{\|z\|} ) \big| \\
&=\|u\| \,\|v\| \,\|w\| \,\|z\| \sum_{ijkl}\alpha_i\beta_j\rho_k\varrho_l  \big|g(R(e_i,e_j)e_k,e_l) \big|\\
& \leq  \alpha_1 ( \beta_1+ \beta_2) ( \rho_1 + \rho_2 + \rho_3) ( \varrho_1 + \varrho_2 + \varrho_3+ \varrho_4)4/3 \tilde{K} \|u\| \,\|v\| \,\|w\| \,\|z\|\, \\
& \leq 1\cdot \sqrt{2} \cdot \sqrt{3} \cdot \sqrt{4} \cdot 4/3 \tilde{K} \|u\| \,\|v\| \,\|w\| \,\|z\|\, \leq 7\tilde{K} \|u\| \,\|v\| \,\|w\| \,\|z\|.
\end{align*} 
This concludes the proof.
\end{proof}

As a direct consequence of the generalized Berger's lemma, if $\cU_p$ is an $\eps$-small normal neighborhood of $p$ with respect to $\T$, using Corollary \ref{cor:geod coor}, we get
\[  |g( R(V_i,V_j) V_k, V_j)| \leq 2\tilde{K}  \qquad \mbox{ and } \qquad   |g(V_j, V_1 - \Xi/ \| \Xi \| )| < \varepsilon   \mbox{ for all $i,j,k \in \{ 1 \ldots n\}$.} \]
Moreover, analogous estimates hold inside $\T(\cU_p)$
\[  |g( R(W_i,W_j) W_k, W_j)| \leq 2\tilde{K} \qquad \mbox{ and } \qquad |g(W_j, W_1 - \Pi(\Xi)/ \| \Xi \| )| < \varepsilon   \mbox{ for all $i,j,k \in \{ 1 \ldots n\}$,}    \]
where the field $\Pi(\Xi) $ is, as usual, obtained by parallel transport along the geodesics between $q$ and $\T(q)$ for every $q \in \mathcal{U}_p$.

\subsection{Construction of perturbed geodesics}

We can now state the main result of this section.

\begin{thm}\label{lemmone}
Let $\M$ be a compact Riemannian manifold with the Ricci curvature bounded from below by $K$, and let again $\tilde{K}$ be a uniform bound for the sectional curvatures.

Let  $\varepsilon > 0$ and $\T \in \textit{Diff}(\Omega,\Delta)$ be a diffeomorphism with infinitesimal generator $\Xi$. For any $p$, consider $\cU_p$ its $\varepsilon$-small normal neighborhood with respect to $\T$ and the respective frames $\{V_i\} \in \Gamma(F_O \cU_p)$ and $\{W_i\} \in \Gamma(F_O \T(\cU_p))$ provided by Definition \ref{small normal neighborhood} (or Remark \ref{rmk:fixed} in case $p$ is a fixed point for $\T$). 
Then for every geodesic $\gamma$ connecting $q \in \cU_p$ to $\q := \T(q) \in \T(\cU_p)$ (with length $L(\gamma)=d(q,\q)$) there exists a family of variations $\{\gamma_i^s\}_i$ satisfying
\begin{equation}\label{eq:initial data}
\gamma_{i}^s(0)= \exp_{q}\big(sV_i(q)\big) \qquad \text{ and }\qquad \gamma_{i}^s(L(\gamma)) = \exp_{\q}\big(sW_i(\q)\big),
\end{equation}
such that 
\begin{equation}\label{stima lemmone sulle variazioni}
\sup_i \left| \frac{d}{ds} \Big|_{s=0} L(\gamma_{i}^s) \right| <  C \varepsilon,   \qquad  \sum_{i=2}^n \frac{d^2}{ds^2}\Big|_{s=0} L(\gamma_{i}^s) <  C \varepsilon - K  L(\gamma) \quad \mbox{ and  } \quad \frac{d^2}{ds^2}\Big|_{s=0} L(\gamma_{1}^s) <  2C \varepsilon ,
\end{equation}

where $C$ is a positive constant  depending only on $K, \tilde{K}$, the dimension and diameter of $\M$ and the Lipschitz constant of $\Xi$ on $\cU_p$.
\end{thm}

The rest of this section is devoted to proving Theorem \ref{lemmone}. 
Let us observe that, given $q \in \cU_p$, it is not true in general that $V_1(q) = \dot{\gamma}(q)$ or $W_1(\T(q)) = \Pi_{\T(q)}V_1(q)$. Indeed, these identities are, a priori, only valid if $q=p$. The mismatch between $\dot{\gamma}(q)$ and $V_1$, and between $W_1(\T(q))$ and $\Pi_{\T(q)}V_1(q)$ comprises the main issue in the proof of Theorem \ref{lemmone}.
On the other hand, whenever $q$ satisfies both
\begin{equation}\label{eqn:casobello} V_1(q) = \dot{\gamma}(q) \quad \mbox{ and } \quad  W_1(\T(q)) = \Pi_{\T(q)} \dot{\gamma}(q)  
\end{equation}
then estimates \eqref{eq:initial data} and \eqref{stima lemmone sulle variazioni} would be straightforward. This lucky situation is investigated in the following Lemma.

\begin{lemma}\label{lem:caso bello}
Let $x,y\in\M$ and denote with $\gamma$ the geodesic 
$$\gamma:[\alpha,\beta]\to\M \qquad \text{ with } \;\gamma(\alpha)=x \text{ and } \gamma(\beta)=y\,.$$
Let $\{e_i (x)\}$ be an orthonormal basis of $T_x\M$ such that $e_1(x) = \dot{\gamma}(\alpha)$ and denote the 
parallel transport of $e_i(x)$ along $\gamma$ by $e_i(\gamma(t))=\Pi_{\gamma(t)} e_i(x)$. Finally, consider the smooth variation of $\gamma$ given by 
$$   f_i : [0,\delta]\times [\alpha,\beta]\to \M,\qquad (s,t)\mapsto f_i(s,t):= \exp_{\gamma(t)}(s\,e_i)
$$
where we omit the dependence of $e_i$ on the point on $\gamma(t)$
and set $\gamma_i^s(\cdot)=f_i(s,\cdot)$. Then we have
$$
\frac{d}{ds}\Big|_{s=0}L(\gamma_i^s) = 0 \qquad \text{ and } \qquad
 \frac{d^2}{ds^2}\Big|_{s=0}L(\gamma_i^s) =    -  \int_\alpha^\beta g\big( R(e_i,\dot{\gamma})  \dot{\gamma},e_i \big) \,dt  
$$
\end{lemma}
\begin{proof}
For any fixed $t\in[\alpha,\beta]$,  $f(\cdot,t)$ is a geodesic generated by the `starting' tangent vector $ e_i$.
By extending $\{e_i\}$ to these geodesics $f_i$ by parallel transport, we obtain
$$ \nabla_{\dot{\gamma}}\,e_i =0 \qquad \text{ and } \qquad \nabla_{e_i} e_i =0 \,.$$ 
By using that the parallel transport preserves the scalar product, i.e.
$$g\big( e_i (\gamma(t)),e_j (\gamma(t))\big)= g\big(\Pi_{\gamma(t)} e_i(x), \Pi_{\gamma(t)} e_j(x) \big) = g\big( e_i(x), e_j(x) \big) $$
and $\{e_i\}$ are orthonormal, by Lemma~\ref{lem:variationslength} we can conclude our proof.
\end{proof}

 As already observed, we cannot apply, in general, Lemma \ref{lem:caso bello} to the points of $\cU_p$. We want to understand now how much the equality \eqref{eqn:casobello} is violated.
 
\begin{lemma}\label{lem:confrontoVW} There exists a constant $C=C(Lip(\Xi|_{\cU_p}) , \tilde{K}, n)$ such that for every $q \in \cU_p$, letting $\textbf{q} =\T(q) \in \T(\cU_p)$, we have 
$$ \|W_i(\textbf{q}) - \Pi_{\T(q)}V_i \| \leq C  d(p,q) \qquad \forall i=1, \ldots, n.$$
In case $p$ is a fixed point for $\T$ we can suppose $V_i=W_i$ and
$$ \|W_i(\textbf{q}) - \Pi_{\T(q)}V_i \| \leq C\varepsilon d(q,\T(q)) \qquad \forall i=1, \ldots, n.$$
\end{lemma}

\begin{proof}
Let us denote $\ell=d(p,\textbf{p})$. At first let us define the geodesic $\sigma:[0,1] \to \M$ connecting $p$ to $q$; then for every $s \in [0,1]$ we consider the geodesic $\gamma_s:[0,1] \to \M$ from $\sigma(s)$ to $\T(\sigma(s))$. Notice that $|\dot{\gamma}_s| = d(\sigma(s),\T(\sigma(s)))\leq d(p,q)+2 \varepsilon$. We then want to bound $J(s,t)=\partial_s \gamma_s(t)$. Since $J$ is a variation of geodesics, it satisfies the Jacobi equation
$$ \frac {D^2J}{dt^2} (s,t) + R(\dot{\gamma}_s(t) , J(s,t))\dot{\gamma}_s(t) =0. $$
Considering the function $f_s(t)=\| J(s,t)\|^2 + \|\nabla_{\dot{\gamma}_s} J(s,t)\|^2$, we get that 
\begin{align*} f'_s(t) &=2 \langle J(s,t), \nabla_{\dot{\gamma}_s} J(s,t) \rangle - 2 \langle \nabla_{\dot{\gamma}_s} J(s,t), R(\dot{\gamma}_s(t) , J(s,t))\dot{\gamma}_s(t)\rangle  \\
& \leq (2+2\tilde{K})|J(s,t)| \cdot | \nabla_{\dot{\gamma}_s} J(s,t)|  \\
& \leq (1+\tilde{K}) \cdot \Bigl(\| J(s,t)\|^2 + \|\nabla_{\dot{\gamma}_s} J(s,t)\|^2 \Bigr) = (1+\tilde{K}) f_s(t).
\end{align*}
In particular $f_s$ increases at most exponentially $f_s(t) \leq e^{(1+\tilde{K})t} f_s(0)$. Notice that $\dot{\gamma}_s(0)= \nabla \varphi (\gamma(s))$ and in particular 
$$  \nabla_{\dot{\gamma}_s} J(s,0)  =  \nabla_{\dot{\gamma}_s} \partial_s \gamma= \nabla_{\partial_s \gamma} \dot{\gamma}_s = \frac D{ds}   \Xi (\gamma(s)).$$
Given that $\Xi$ is $L$-Lipschitz in $\cU_p$ we obtain $| \nabla_{\dot{\gamma}_s} J(s,0)| \leq L |\partial_s \gamma_s(0)|=L d(p,q)$. Since $|J(s,0)|=|\partial_s \gamma_s(0)|=d(p,q)$ we have finally that $|J(t,s)| \leq \sqrt{f_s(t)} \leq C \cdot d(p,q)$ for every $t,s$. Now we can consider for every $s$ the transported field on $\gamma_s$ defined as $V_i(t,s)=\Pi_{\gamma_s(t)} (V_i(0,s))$.

Notice that we want to understand how far is $V_i(1,1)$ from $W_i(\textbf{q})$. In order to do this we consider $h(s)=V_i(1,s)-W_i(\gamma(1,s))$: notice that we wanto to estimate $h(1)$, and we know that $h(0)=0$ by definition of $W_i$ in $\gamma(1,0)=\textbf{p}$. In particular, we have 
$$  \|W_i(\textbf{q}) - \Pi_{\T(q)}V_i \| = |h(1)|\leq |h(0)| + \int_0^1| \partial_s h (s) | \, ds;$$
But now, using the properties of $W_i$ as a local frame for geodesic normal coordinates we have that in charts
$$ \|\nabla_VW_i\| = \Bigl\|\sum_{j,k=1}^n v_j\Gamma_{i,j}^kW_k \Big\| \leq n^2\|V\| \varepsilon \quad \text{for a generic } V = (v_1, \ldots, v_n);   $$
a similar estimate holds for $V_i$. In particular, we only now need to estimate 
$\partial_s V_i(1,s) = \nabla_{\partial_s \gamma} V_i(1,s)$, and in order to do so we use that

$$\nabla_{\partial_s \gamma} V_i(1,s) = \nabla_{\partial_s \gamma} V_i(0 ,s)+ \int_0^1 \nabla_{\dot{\gamma}_s} \nabla_{\partial_s \gamma} V_i(t ,s) \, dt.$$
In particular, using the Riemann tensor and that $V_i(t,s)$ is transported parallel along the curves $\gamma_s$ for every $s$, we get
 $$\int_0^1 \nabla_{\dot{\gamma}_s} \nabla_{\partial_s \gamma} V_i(t ,s) \, dt = \int_0^1  \nabla_{\partial_s \gamma}\nabla_{\dot{\gamma}_s} V_i(t ,s) + R(\partial_s \gamma, \dot{\gamma}_s)V_i(t,s) \, dt = \int_0^1 R(\partial_s \gamma, \dot{\gamma}_s)V_i(t,s) \, dt$$
Now by the (generalized) Berger's Lemma~\ref{rmk: per concludere il lemmone} we obtain that 
\begin{align*} \|\partial_s h(s)\| &\leq \| \nabla_{\partial_s \gamma} W_i\| +\|\nabla_{\partial_s \gamma} V_i(0 ,s)\| + \Bigl\| \int_0^1 R(\partial_s \gamma, \dot{\gamma}_s)V_i(t,s) \, dt \Bigr\| \\
& \leq n^2\varepsilon \| \partial_s \gamma(1,s)\|   + n^2\varepsilon \| \partial_s \gamma (0,s)\| + 4\tilde{K} \int_0^1 \| \partial_s \gamma \| \| \dot{\gamma}_s\| \| V_i(t,s)\| \, dt.
\end{align*}
At this point we can use that $\|\partial_s \gamma_s\| \leq C \cdot d(p,q)$, $\| \dot{\gamma}_s \| \leq d(p,\textbf{p}) + 2 \varepsilon$, $\|V_i(t,s)\|=1$, to conclude 
$$ \|W_i(\textbf{q}) - \Pi_{\T(q)}V_i \| \leq  \int_0^1 \| \partial_s h(s)\| \,ds  \leq C \cdot d(p,q)$$
for some constant $C$ depending only on $\cU_p$. 

As for the fixed point case, we know we can suppose $V_i=W_i$ and that the whole geodesic $\gamma$ between $q$ and $\T(q)$ is inside $\cU_p$. In particular we have $\|W_i(\textbf{q}) - \Pi_{\T(q)}V_i\|\leq \int_0^1 \| \nabla_{\dot{\gamma}(t)} V_i \| \, dt \leq n^2 \varepsilon L(\gamma) = n^2 \varepsilon d(q,\T(q))$.
\end{proof}

\begin{definition}[Quasi-orthonormality]\label{def:qob} Let $V$ be a $n$-dimensional Hilbert space and $0 \leq \sigma < 1/n$. We say that $\{X_i\}_i$ is a $\sigma$-orthonormal basis if $  | \langle X_i, X_j \rangle - \delta_{i,j} | \leq \sigma$.
\end{definition}

\begin{lemma} Let $\{X_i\}_i$ that be a $\sigma$-orhtonormal basis for $V$, and $n$-dimensional Hilbert space.
Then, as long as $\sigma n < 1/2$, there exists a dimensional constant $C_n$ such that for any linear map $A: V \to V$ it holds
\begin{equation}\label{eqn:stimatracciaA}
 \left| \tr(A)  - \sum_{i=1}^n \langle AX_i, X_i \rangle \right|  \leq C_n \sigma \|A\|.
\end{equation}
 \end{lemma}
 
 \begin{proof}
First of all the condition $\sigma < 1/n$ grants us that $X_i$ is in fact a basis thanks to the fact that the matrix  $X^TX$ is diagonally dominant and thus invertible. Now, consider any linear map $A:V \to V$ and let $A_{i,j}$ be the matrix of $A$ with respect to $X$. Notice that $|\langle AX_i , X_{i'} \rangle | = |\sum_j A_{i,j} \langle X_j, X_{i'} \rangle  |\geq |A_{i,i'} | - \sigma \sum_{j} | A_{i,j} |$; summing up this inequality on $i'$ we get
\begin{equation}\label{eqn:stimaoffdiagonalA} 
(1-n\sigma) \sum_j |A_{i,j}| \leq \sum_{j} |\langle AX_i , X_{j} \rangle| \leq n \| A\| (1+\sigma).
\end{equation}
In the end we can estimate $\tr(A)=\sum_i A_{i,i}$ using estimate \eqref{eqn:stimaoffdiagonalA} on the remainder:
$$
 \left| \tr(A)  - \sum_{i=1}^n \langle AX_i, X_i \rangle \right| = \left| \sum_{i,j} | A_{i,j} | |\langle X_i, X_j \rangle - \delta_{i,j} | \right| \leq \frac{\sigma (1+\sigma) n^2}{1-n\sigma} \|A\|,
$$
and so we can conclude, for example choosing $C_n=4n^2$.

\end{proof}

\begin{prop}\label{prop:rotazione}
Let $\M,\, \eps,\, p,\, \T,\,\cU_p,\, \{V_i\},\, \{W_i\}$ be as Theorem \ref{lemmone}. Consider $q \in \cU_p$, $\gamma$ the geodesic connecting $q$ to $\T(q)$ and, for simplicity, denote $\ell = L(\gamma)$.
Then it is possible to construct $\{X_i\}\in F_O\M|_\gamma$ for which the following holds:
\begin{itemize}
\item[(I)] $X_i(\gamma(0))=V_i(q)$  and $X_i(\gamma(\ell))=W_i(\T(q))$ for every $i\in\{1,\dots,n\}$.
\item[(II)] The variations of $\gamma$ given by
$$f_i : [-\delta,\delta]\times [0,\ell]\to \M,\qquad (s,t)\mapsto f_i(s,t) = \gamma_i^s(t) = \exp_{\gamma(t)}\Big(s\, X_i(\gamma({t}))\Big) 
$$
satisfy
$$ \frac{d}{ds}\Big|_{s=0}L(\gamma_i^s) = g(W_i(\T(q),\dot{\gamma}(\ell)) - g(V_i(q),\dot{\gamma}(0)) \quad $$
$$ \frac{d^2}{ds^2}\Big|_{s=0}L(\gamma_i^s) \leq  C^2\varepsilon - \int_{0}^{\ell} g\big( R(X_i,\dot{\gamma})  \dot{\gamma},X_i \big) \,d\tau  $$.
\item[(III)] Moreover, $X_i$ is a $2C\varepsilon$-orthonormal basis in the spirit of Definition \ref{def:qob}, that is it satisfies 
\begin{equation*} \label{eq:epsortho} |g(X_i(\gamma(t)),X_j(\gamma(t)))- \delta_{i,j} | \leq 2C\varepsilon \qquad \forall i,j \in \{1, \ldots, n\} , t \in [0,\ell] 
\end{equation*}

\end{itemize} 
\end{prop}
\begin{proof}

Let $\eta : [0,\ell]\to [0,1]$ be a smooth function such that $\eta(0)=0$ and $\eta(\ell)=1$:
we can also assume $|\eta'(\tau)| \leq \frac 2{\ell} $.
Now we consider $\tilde{V}_i$ and $\tilde{W}_i$ as parallel transport along $\gamma$ respectively  of $V_i(\gamma(0))$ and $W_i(\gamma(\ell))$ along $\gamma$. In particular, thanks to Lemma \ref{lem:confrontoVW} and condition (ii) of Definition \ref{small normal neighborhood} we have $\|\tilde{W}_i(\gamma(\tau))-\tilde{V}_i(\gamma(\tau))\| \leq C \varepsilon $ for some $C=C(Lip(\nabla \phi|_{\cU_p}) , \tilde{K}, n)$ and every $\tau \in [0,\ell]$. Then we define the vector fields $X_i$ as 
\begin{equation*}\label{eq: def X_i}
 \Gamma(F_0\M|_{\gamma}) \ni X_i(\gamma(\tau)):=\big(1-\eta(\tau) \big) \tilde{V}_i(\gamma(\tau)) +  \eta(\tau) \,\tilde{W}_i(\gamma(\tau))  \qquad  \text{ for } \tau\in [0,\ell].
\end{equation*}
Notice that since $\nabla_{\dot{\gamma}}\tilde{V}_i=\nabla_{\dot{\gamma}}\tilde{W}_i=0$ we have $\nabla_{\dot{\gamma}}X_i = \eta' (\tilde{W}_i - \tilde{V}_i)$ therefore $\| \nabla_{\dot{\gamma}} X_i\| \leq \frac {2 C\varepsilon}{\ell}$. 

Plugging this estimate in \eqref{2nd variation} and using \eqref{eqn:separation} we obtain the estimate of the second derivative of the length $L(\gamma_s)$ in $(II)$ (notice that the variation curves have been constructed in such a way that, when extended properly outside, we have $\nabla_{X_i} X_i=0$ on $[0,\ell]$). Next, simply using the definition of $X_i$ in \eqref{1st variation} we obtain its first derivative. Notice that a similar calculation can be carried on for the fixed point case, using the corresponding estimate in Lemma \ref{lem:confrontoVW}.

In order to prove the quasi-orthonormality, we first show it for $i=j$, where we observe 
$$g(X_i,X_i)= \eta(t) g(\tilde{V}_i, \tilde{V}_i)+(1-\eta(t))g(\tilde{W}_i, \tilde{W}_i) - \eta(t)(1-\eta(t))\| \tilde{V}_i-\tilde{W}_i\|^2,$$
and in particular $1-\varepsilon^2 \leq g(X_i,X_i)\leq 1$. For the off-diagonal term we then compute
$$ \frac D{dt} g(X_i,X_j) = g( \nabla_{\dot{\gamma}} X_i , X_j) + g(X_i, \nabla_{\dot{\gamma}} X_j) \leq \|\nabla_{\dot{\gamma}} X_i\| \| X_j\| + \|\nabla_{\dot{\gamma}} X_i\| \| X_j\| \leq  \frac{4 C \varepsilon}{\ell}, $$
where we used $\|X_i\| =g(X_i,X_i)^{1/2} \leq 1$ and $\|\nabla_{\dot{\gamma}} X_i\| \leq \frac { 2 C\varepsilon}{\ell}$. Using that $g(X_i(\gamma(t)),X_j(\gamma(t)))=0$ for $t=0$ and $t=\ell$ we deduce that $|g(X_i,X_j)| \leq 2C \varepsilon$ concluding the proof.
\end{proof}

We now have all the ingredients to prove the main Theorem of this section.

\begin{proof}[Proof of Theorem~\ref{lemmone}]
We discuss in detail the proof of the more difficult case when $p$ is not a fixed point for $\T$, while we refer to Remark \ref{rmk: punti fissi di T} for the suitable modifications of the following argument in the case of a fixed point.
Let $q \in \cU_p$ be fixed and $\gamma$ be the length minimizing geodesic connecting $q$ to $\T(q)$ inside $\M$. We consider the geodesic parametrized by the arc-length, so $\gamma: [0,\ell] \to \M$ where $\ell$ denotes the length of $\gamma$. In what follows we will use the same notation as in Proposition \ref{prop:rotazione}. Let then $X_i$ be the class of variation fields introduced in Proposition \ref{prop:rotazione}, where $\{V_i\}$ and $\{W_i\}$ are the usual coordinate frames associated to the $\eps$-small normal neighborhood $\cU_p$. 
Accordingly to Proposition \ref{prop:rotazione}, we then define the variations of $\gamma$ 
$$f_i : [0,\delta/2]\times [0,\ell]\to \M,\qquad (s,t)\mapsto f_i(s,t):= \exp_{\gamma(t)}\Big(s\, X_i(\gamma({t}))\Big) \qquad \gamma_i^s(\cdot) = f_i(s,\cdot),
$$
where $\delta$ is provided by the definition of $\varepsilon$-small normal neighborhood and, hence, it ultimately depends on $\varepsilon$. 
Condition \eqref{eq:initial data} follows immediately by Proposition \ref{prop:rotazione} $(I)$. 

Next, gathering together \eqref{1st variation}, the first order identity in Proposition \ref{prop:rotazione} $(II)$, and the fact that $\Pi_{\T(q)} \dot{\gamma}(0) = \dot{\gamma}(\ell)$ we get
\begin{align*}
 \sup_i \left| \frac{d}{ds} |_{s=0} L(\gamma_i^s) \right| &= \sup_{i} \left| g(W_i(\T(q)), \dot{\gamma}(\ell) - g(V_i(q), \dot{\gamma}(q)) \right|  \\ &= \sup_i \left| g(W_i(\T(q)), \dot{\gamma}(\ell) - g(\Pi_{\T(q)} V_i(q), \Pi_{\T(q)} \dot{\gamma}(0)) \right| \\
 &=  \sup_i \left| g(W_i(\T(q)) - \Pi_{\T(q)} V_i(q) , \dot{\gamma}(\ell) \right| \leq  C  \varepsilon d(q,\T(q)), \end{align*}
 where in the end we used Lemma \ref{lem:confrontoVW}: we conclude that the first order inequality in \eqref{stima lemmone sulle variazioni} holds.

We next notice that the lower bound on the Ricci curvature implies that if $V_i$ is an orthonormal basis we have
$$ - \sum_{i=1}^n g\big( R(V_i,\dot{\gamma})  \dot{\gamma},V_i \big) = - \sum_{i=1}^n g\big( R(\dot{\gamma},V_i) V_i, \dot{\gamma} \big) =-  g\big( \Ric_p(\dot{\gamma}), \dot{\gamma} \big)    = ric(\dot{\gamma}, \dot{\gamma})\leq - K g(\dot{\gamma},\dot{\gamma})  \,.$$
Notice that $ric(v,v)= -\tr(A)$ where $A: T_pM \to T_pM$ is the linear map defined by $AX:= R(X,v)v$; in particular, thanks to Lemma~\ref{rmk: per concludere il lemmone} we have $\|A\|\leq 7 \tilde{K} \|v\|^2$ where $\tilde{K}$ is a bound on the sectional curvature and so, using that  $X_i$ is a $2C \varepsilon$-orthonormal base we get, by \eqref{eqn:stimatracciaA}
\begin{equation}\label{eqn:qo}
 \left| - ric(\dot{\gamma}, \dot{\gamma})- \sum_{i=1}^n g\big( R(X_i,\dot{\gamma})  \dot{\gamma},X_i \big) \right| \leq 14 nC \cdot C_n \varepsilon \tilde{K}. 
\end{equation}
Notice now that $\|\dot{\gamma} - X_1\| \leq C \varepsilon$, and in particular by the antisymmetry of the Riemann tensor and Lemma~\ref{rmk: per concludere il lemmone} we have $|g(R(X_1,\dot{\gamma}),\dot{\gamma},X_1) | \leq 7\tilde{K} C^2\varepsilon^2$.So applying now \eqref{eqn:qo} to estimates $(II)$ of Proposition~\ref{prop:rotazione}, we get 
\begin{align*}
\sum_{i=2}^n \frac{d^2}{ds^2}|_{s=0} L(\gamma_i^s) <& nC \varepsilon + \sum_{i=2}^n \int_0^{\ell} g\big( R(\dot{\gamma}, X_i) \dot{\gamma} , X_i\big) d\tau  \\
\leq & nC \varepsilon  - K  L(\gamma) + n C \cdot C_n \varepsilon  \tilde{K} L(\gamma) + 7 \tilde{K} C^2 \varepsilon^2 L(\gamma)  \\
\leq & -K L(\gamma) + nC\varepsilon ( 1 + 2\tilde{K}C_n L(\gamma)),
\end{align*}
where estimate $C$ is depending only on the Lipschitz constant of $\Xi$ in $\cU_p$ and $C_n$ depends only on the dimension $n$. Similarly we get $ \frac{d^2}{ds^2}|_{s=0} L(\gamma_1^s) \leq 2C\varepsilon $.
\end{proof}

\begin{rmk}\label{rmk: punti fissi di T}
Let $\varepsilon >0$ be fixed. If $p$ is a fixed point for $\T$, then we can easily obtain
$$\left| \frac{d}{ds} \Big|_{s=0} L(\gamma_{i}^s) \right| <  C \varepsilon   \quad \mbox{for every $i$, and } \quad  \sum_{i=1}^n \frac{d^2}{ds^2}\Big|_{s=0} L(\gamma_{i}^s) <   (  C \varepsilon - K) L(\gamma) \,$$
where $C$ is the same constant of Theorem \ref{lemmone}.
Indeed we can find a $\varepsilon$-small normal neighborhood $\mathcal{V}_p$ of $p$ with respect to $\T$ and some $\cU_p \subset \mathcal{V}_p$ such that also $\T(\cU_p) \subset \mathcal{V}_p$. Then, by choosing $W_i = V_i$ and using that $d(q,\T(q))$ is very small, estimates \eqref{eq:initial data} and \eqref{stima lemmone sulle variazioni} are still valid in this setting. 
\end{rmk}

We conclude this section with an immediate consequence of Theorem~\ref{lemmone} and Remark~\ref{rmk: punti fissi di T}.

\begin{cor}\label{corollario che serve per 4gi}
Let $\T \in \mathit{Diff}(\Omega, \Delta)$ with infinitesimal generator $\Xi$ and $p \in \Omega$, where $\Omega, \Delta $ are open subsets of $\M$. For every $\varepsilon>0$ there exists an $\varepsilon$-small normal neighborhood $\cU_p$ such that the following holds. 
For every $q \in \cU_p$ and $s \in [0,\delta/2]$ one can define 
$$q^{\pm s}_i:= exp_q( \pm sV_i(q)) \qquad \text{ and } \qquad  \textbf{q}^{\pm s}_i:= exp_{\T(q)}( \pm s W_i(\T(q)))\,,$$ 
where $W_i = V_i$ in case $p$ were a fixed point for $\T$.
Denoting as usual $\gamma$ the geodesic connecting $q$ and $\q= \T(q)$ and $\ell = L(\gamma)$, for every $s \in [0,\delta/2]$ and every $i \in \{ 1 \ldots n\}$ there exist suitable variations of $\gamma$
$$\gamma_i^{\pm s}:[0,\ell]\to \M \qquad \text{ such that } \qquad \gamma_i^{\pm s}(0)=q^{\pm s}_i \quad  \text{ and } \quad \gamma_i^{\pm s}(\ell)=\textbf{q}^{\pm s}_i$$ 
satisfying estimates \eqref{stima lemmone sulle variazioni}. Moreover, for every non-negative increasing $h\in C^{2}(\M)$ one has
\begin{equation}\label{stima lemmone}
\sum_{i = 2}^n [h(d(q^{+ s}_i, \textbf{q}^{+ s}_i)) + h(d(q^{- s}_i, \textbf{q}^{- s}_i)) - 2h(d(q, \textbf{q}))] \leq s^2  (C \eps - KL(\gamma)) \, h'(d(q,\textbf{q})) + o(s^2) + s^2 \varepsilon^2 \tilde{C}
\end{equation}
\begin{equation}\label{stima lemmone2}
h(d(q^{+ s}_1, \textbf{q}^{+ s}_1)) + h(d(q^{- s}_1, \textbf{q}^{- s}_1)) - 2h(d(q, \textbf{q})) \leq 2s^2 C \eps \, h'(d(q,\textbf{q})) + o(s^2) + s^2 \varepsilon^2 \tilde{C}
\end{equation}

where $C$ depends only $n$, $K$, $\tilde{K}$ and $\tilde{C}$ depends on $n$ and  $\sup_{[0,diam \M]} h''$. 
\end{cor}
\begin{proof}
The proof is immediate once we observe that the curve $\gamma_i^{\pm s}$ connecting  $q^{\pm s}_i$ and $\textbf{q}^{\pm s}_i$ is surely longer than $d(q^{\pm s}_i, \textbf{q}^{\pm s}_i)$. We recall that, by construction, $\gamma$ is precisely the geodesic between $q$ and $\q= \T(q)$. Therefore,
\begin{equation}\label{1}
d(q^{\pm s}_i, \textbf{q}^{\pm s}_i)   \leq L(\gamma_i^{\pm s})\,, \qquad \mbox{and} \qquad d(q, \q) = L(\gamma),
\end{equation}  
and the conclusion follows by gathering together \eqref{1} and \eqref{Taylor exp con h}.
\end{proof}

\section{Linear and nonlinear gradient inequalities}

\subsection{The smooth case}

As a first step, we shall prove, in the smooth setting, the validity of the 4 gradient inequality on small enough neighborhoods of an arbitrary point outside a negligible set $\Sigma$.
In what follows we will use the notation introduced in Theorem \ref{lemmone}.

\begin{prop}\label{lemma 4gi su intornino}
Assume the Setup~\ref{setup}, where additionally $h$ and $\ell$ respectively satisfy (cost2reg) and (cvxreg). Let $\phi,\psi$, the optimal Kantorovich potentials between $\mu$ and $\nu$ for cost $c=h\circ d$ and $\T$ be the optimal map from $\mu$ to $\nu$.
Then there exists a $\vol$-negligible closed set $\Sigma$ such that for every $p \in \M \setminus \Sigma$ there exists a constant $c_p>0$ depending on the geometric setting of the problem (thus on $n,\mu,\nu,h,\phi,\psi$ and the geometry of $\M$) such that for every  $0<\eps <1$ one can find an $\eps$-small normal neighborhood of $p$, denoted by $\cU_p$, for which the following inequality holds

\begin{equation}\label{stima 4gi su intornino}
 \int_U \mu \,\div \nabla  \phi\, \vol + \int_{\T(U)}  \nu \,\div \nabla \psi\, \vol 
\leq -K \int_U h'\big(d(q,\T(q))\big)d(q,\T(q)) d\mu + \eps c_p  \mu(U)
\end{equation}
whenever $U$ is a geodesic open ball contained in $\cU_p$. If moreover $p$ is not a fixed point for $\T$ we also have the stronger inequality

\begin{equation}\label{stima 4gi su intorninplus}
\begin{aligned}
 \int_U \mu \,\div \nabla  \phi\, \vol &+ \int_{\T(U)}  \nu \,\div \nabla \psi\, \vol \\
 &\leq -K \int_U h'\big(d(q,\T(q))\big)d(q,\T(q)) d\mu + I_1(U) +\eps c_p  \mu(U),
  \end{aligned}
\end{equation}
where $I_1(U):=\int_U  \mu \cdot \nabla_V \nabla_V \phi \, \vol +   \int_{\T(U)}  \nu \cdot \nabla_W \nabla_W \psi \, \vol $, in which $V= \frac{\nabla \phi}{ \| \nabla \phi \|} $ and $W= \frac{\nabla \psi}{ \| \nabla \psi \|}$.  Moreover $I_1(U)$ satisfies 
\begin{equation}\label{eqn:resto4gi}
I_1(U) \leq c_p \varepsilon \mu(U)\,.
\end{equation}
\end{prop}

\begin{proof} Let us consider $\Sigma=\Sigma_\mu \cup \Sigma_\nu$ given by Theorem~\ref{thm:DF} and $c=h \circ d$. In particular we have $\vol(\Sigma)=0$ and $\T$ is at least a $C^3$ diffeomorphism in $\M \setminus \Sigma$;  moreover, for the infinitesimal generator of $\T$ we have $\Xi = \lambda^{-1} (-\nabla \phi) \in C^2 (\M \setminus \Sigma)$ which implies that $\phi \in C^3(\M \setminus \Sigma)$; 

Let $p,\,\eps$ be fixed and let $\cU_p$ be an $\eps$-small normal neighborhood of $p$ with respect to $\T$. Observe, that the existence of such $\cU_p$ is ensured by Lemma \ref{lemma: esiste sempre small normal neighborhood}, moreover estimates of Corollary \ref{corollario che serve per 4gi} applies in this setting. We introduce the fields $\{\pm V_i\},\,\{\pm W_i\}$ given by Definition \ref{small normal neighborhood}
and the associated flows
\[  \left\lbrace \begin{array}{ll}
\partial_s \Phi^{\pm} (s,q) = \pm V_i(\Phi^\pm(s,q))\\
\Phi^\pm_i(0,q)= \pm V_i(q)
\end{array} \right.  \qquad   
 \left\lbrace \begin{array}{ll}
\partial_s \Psi^{\pm} (s,\T(q)) = \pm W_i(\Psi^\pm(s,\T(q)))\\
\Psi^\pm_i(0,\T(q))= \pm W_i(\T(q))
\end{array} \right.
 \]
 for every $q \in \cU_p$ and $s \in [0,\delta]$, where $\delta$ is given by the small normal neighborhood, so it ultimately depends on $p$ and $\eps$. Moreover, we let $\Phi^\pm_i = \Psi^\pm_i$ whenever $p$ is a fixed point for $\T$.
 For future use, we calculate 
 \[  \Phi^+_i (s,\cdot)_\sharp \vol = \vol + s \, (\div \, V_i) \vol  + o(s)  \]
which, thanks again to Corollary \ref{cor:geod coor} and 
$$ \div(V_i)=\sum_{j=1}^n g(\nabla_{V_j}V_i,V_j)= \sum_{j=1}^n \Gamma_{ij}^j \,, $$
 implies
\begin{equation}\label{scarto pushforward}
| \Phi^+_i (s,\cdot)_\sharp \vol - \vol  | \leq s\eps n + o(s) 
\end{equation}
for all $s 	\in [0,\delta]$. Analogous estimates hold also for $\Phi^-_i$ and $\Psi^\pm_i$.
Moreover, from \eqref{da trasporto ottimo 1} and \eqref{da trasporto ottimo 2} we deduce 
\begin{equation*}\label{da trasporto 1 e 2 bis}
\begin{split}
&\phi\big( \Phi^\pm_i (s,q)  \big) + \psi\big( \Psi^\pm_i (s,\T(q))  \big)   \leq h \big( d (\Phi^\pm_i (s,q), \Psi^\pm_i (s,\T(q))) \big) \\
&\phi (q) + \psi(\T(q)) = h\big(d(q,\T(q))\big)
\end{split}
\end{equation*}
for every $(s,q) \in [0,\delta] \times \cU_p$.
Let now $U \subset \cU_p$ be an open geodesic ball arbitrary fixed, then \eqref{stima lemmone} and \eqref{stima lemmone2} respectively implies 
\begin{align}\label{da sopra 4gi}
\nonumber
\sum_{i=2}^n \int_U &\phi(\Phi^+_i (s,q)) + \phi(\Phi^-_i(s,q))  + \psi(\Psi^+_i (s,\T(q))) + \psi(\Psi^-_i(s,\T (q))) - 2 \big(\phi(q) + \psi(\T(q))\big) d\mu \\
\leq & \,  -s^2K \int_U h' \big( d(q,\T(q)) \big) d(q,\T(q)) d\mu  + s^2 \varepsilon(C+\varepsilon \tilde{C})\vol(U) + o(s^2)
\end{align}
\begin{align}\label{da sopra 4gi2}
\nonumber
 I_1^s := \int_U &\phi(\Phi^+_1(s,q)) + \phi(\Phi^-_1(s,q))  + \psi(\Psi^+_1(s,\T(q))) + \psi(\Psi^-_1(s,\T (q))) - 2 \big(\phi(q) + \psi(\T(q))\big) d\mu \\
\leq &  s^2 \varepsilon(C+\varepsilon \tilde{C})\vol(U) + o(s^2)
\end{align}

where $C,\,\tilde{C}$ are geometric constants provided by Corollary \ref{corollario che serve per 4gi}. Let us underline that $\tilde{C}$ do not depend on $p$ nor on $U$, while $C$ depends on $p$ and $U$ through the Lipschitz constant of $\Xi$, which is still controlled on compact sets of $\M \setminus \Sigma$.
On the other hand, using \eqref{scarto pushforward}, the uniform bounds of $\phi,\,\psi$ and $\mu,\,\nu$ and the relation $\nu = \T_\sharp \mu$, for every $i = 1 \ldots n$ it is easy to get 
\begin{align}\label{da sotto 4gi}
\nonumber
\int_U &\phi(\Phi^+_i (s,q)) + \phi(\Phi^-_i(s,q))  + \psi(\Psi^+_i (s,\T(q))) + \psi(\Psi^-_i(s,\T (q))) - 2 \big(\phi(q) + \psi(\T(q))\big) d\mu \\
\nonumber
\geq &\, - \int_{U \,\cap \,\Phi^-_i(U)}  \big( \phi(\Phi^+_i (s,q)) - \phi(q)\big) \big( \mu(\Phi^+_i (s,q)) - \mu(q) \big) \vol \\
\nonumber
&\, - \int_{\T(U) \cap \Psi^-_i (\T(U))}  \big( \psi(\Psi^+_i (s,\q)) - \psi(\q)\big) \big( \nu(\Psi^+_i (s,\q)) - \nu(\q) \big) \vol \\
&\, - (2s^2\eps n + o(s^2)) \big( \|  \nabla \phi \,\mu \|_{L^\infty} \vol(U) + \|  \nabla \psi\, \nu \|_{L^\infty} \vol(\T(U)) \big)  + \, \text{boundary terms}
\end{align}
where the \emph{boundary terms} correspond to 
\begin{align*}
&\int_{U \setminus \Phi^-_i(U)} \big( \phi(\Phi^+_i (s,q)) - \phi(q)\big)  \mu (q) \vol - \int_{\Phi^-_i(U) \setminus U} \big( \phi(\Phi^+_i (s,q)) - \phi(q)\big)  \mu(\Phi^+_i (s,q)) \vol \\
+ &\int_{\T(U) \setminus \Psi^-_i(\T(U))}  \big( \psi(\Psi^+_i (s,\q)) - \psi(\q)\big)  \nu (\q) \vol - \int_{\Psi^-_i(\T(U)) \setminus \T(U)} \big( \psi(\Psi^+_i (s,\q)) - \psi(\q)\big)  \nu( \Psi^+_i (s,\q)) \vol.
\end{align*}
Observe that  $\partial U$ and $\partial \T(U)$ being smooth curves and using standard geometric arguments it is possible to deduce
\[  \lim_{s \to 0} \frac{1}{s^2} \, \text{ boundary terms} = \int_{\partial U} \nabla_{V_i} \phi(q) n_i(q) \mu(q) +  \int_{\partial \T(U)} \nabla_{W_i} \psi(\q) \textbf{n}_i(\q) \nu(\q),  \]
where we denoted with $n_i,\,\textbf{n}_i$ the $i$-th components of the outer unit normals to $\partial U$ and $ \partial \T(U)$ respectively. 
Gathering together \eqref{da sopra 4gi}, \eqref{da sotto 4gi}, dividing by $s^2$ and taking the limit as $s \to 0$, we then obtain 
\begin{align*}
& - \int_U \nabla \phi \cdot \nabla \mu\, \vol - \int_{\T(U)} \nabla \psi \cdot \nabla \nu\, \vol + \int_{\partial U} \nabla \phi \cdot n \mu + \int_{\partial \T(U)} \nabla \psi \cdot \textbf{n} \nu  \\
\nonumber
\leq &  - K \int_U h' \big( d(q,\T(q)) \big) d(q,\T(q)) d\mu +\varepsilon \Big[\vol(U)( \eps \tilde{C} + n^2 \| \nabla \phi \, \mu \|_{L^\infty}) + \vol(\T(U)) n^2 \| \nabla \psi \, \nu \|_{L^\infty}\Big].
\end{align*}
Finally, by applying the divergence theorem for Riemannian manifolds, we conclude \eqref{stima 4gi su intornino} with
\[  c_p = \max \left\lbrace C \sup_{\M\times \M} h'(d)d + \eps \tilde{C} +  n^2 \| \nabla \phi \, \mu \|_{L^\infty};\,\,   n^2 \| \nabla \psi \, \nu \|_{L^\infty}   \right\rbrace. \]
Since  we choose $\varepsilon<1$, $c_p$ does not depend on the choice of $U$ and so this concludes the proof. 

Let us assume now that $p$ is not a fixed point. In particular $\nabla \phi \neq 0$ and thus $\Xi \neq 0$. Notice that in \eqref{da sopra 4gi2} we could treat $I_1^s$ differently.  Using that $ \nabla \phi / \| \nabla \phi\| = -\Xi / \| \Xi \|$ since they have the same direction and opposite verse, the fact that  $\|V_1- \Xi / \| \Xi \| \| \leq C \varepsilon$, and the definition of $I_1$, we get 
\begin{align*}
 |I^s_1  - s^2 I_1(U)|  &=  s^2\left|\int_U  \mu \cdot \nabla_{V_1} \nabla_{V_1} \phi \, \vol +    \int_{\T(U)}  \nu \cdot \nabla_{W_1} \nabla_{W_1}  \psi \, \vol  - I_1(U)\right| + o (s^2) \mu(U) \\
 & \leq C \varepsilon s^2 \mu(U) (\| D^2\phi\|_{\infty,U} +  \| D^2\psi\|_{\infty,\T(U)})  + o(s^2) \mu(U).
 \end{align*}
 Dividing by $s^2$ we get $I_1(U) - C \varepsilon \mu(U) \leq \lim_{s \to 0} I_1^s/s^2 \leq I_1(U) + C \varepsilon \mu(U)$. Using again \eqref{da sopra 4gi2}, we can obtain $I_1(U) \leq C \varepsilon \mu(U)$ and using instead the estimate from above of $I_1^s $ in terms of $I_1$ we get the improved inequality \eqref{stima 4gi su intorninplus}. 
\end{proof}

An immediate consequence of Proposition~\ref{lemma 4gi su intornino} is the following 

\begin{prop}\label{corollario a 4gi su intornino}
Assume the Setup~\ref{setup} , where additionally $h$ and $\ell$ respectively satisfy (cost2reg) and (cvxreg). Let $\phi,\psi$, the optimal Kantorovich potentials between $\mu$ and $\nu$ for cost $c=h\circ d$ and $\T$ be the optimal map from $\mu$ to $\nu$. Let $\Sigma$ be the set given by Proposition~\ref{lemma 4gi su intornino}.

Let $\alpha, \beta$ be defined in $\M \setminus \Sigma$ by $\alpha = \frac{\ell' (|\nabla \phi|)}{|\nabla \phi|}$, $\beta = \frac{\ell' (|\nabla \psi|)}{|\nabla \psi|}$ (notice that we can extend $\alpha=0$ where $\nabla \phi=0$ thanks to (cvxreg)). Then 
\begin{equation}\label{stima 4gi}
 \, \div \nabla \phi (q)\,+\,  \, \div\nabla \psi  \big(\T(q)\big)\, \leq -K h'(d(q,\T(q))) d(q,\T(q)) \qquad \forall q \in \M \setminus \Sigma
\end{equation}
and
\begin{align}\label{stima 4gi pesata}
 \, \div (\alpha \nabla \phi) (q)\,+\,  \, \div (\beta\nabla \psi)  \big(\T(q)\big)\, \leq -K \ell'( h'(d(q,\T(q)))) d(q,\T(q))  \qquad \forall q \in \M \setminus \Sigma.
\end{align}
\end{prop}

\begin{proof}
Given $p \in \M \setminus \Sigma$ and $0<\eps<1$ arbitrary small, thanks to Proposition \ref{lemma 4gi su intornino},  we know that there exists some $\eps$-small normal neighborhood $\U_p$ of $p$ where \eqref{stima 4gi su intornino} holds for any geodesic ball $U \subset \U_p$.  Thanks to the regularity assumptions we have that $\div \nabla \phi (q)$ and $\div\nabla \psi  \big(\textbf{q}\big)$ are continuous, respectively in $\cU_p$ and $\T(\cU_p)$, as a consequence of Theorem~\ref{thm:DF}.  In particular,  dividing \eqref{stima 4gi su intornino} by $ \mu(U)= \nu(\T(U))$ (by the injectivity of $\T$) we deduce, by the continuity of the integrands, that 
\[  \div \nabla \phi (q)\,+\,  \, \div\nabla \psi  \big(\T(q)\big)\, \leq  -K h'(d(q,\T(q))) d(q,\T(q)) + c_p\eps.  \]
We conclude \eqref{stima 4gi} by the arbitrariness of $\eps$. In a similar fashion, whenever $p$ is not a fixed point, starting from \eqref{stima 4gi su intorninplus} and \eqref{eqn:resto4gi} we obtain the stronger inequalities
\begin{equation}\label{stima 4gistrong}
  \div \nabla \phi (q)\,+\,  \, \div\nabla \psi  \big(\T(q)\big)\, \leq  -K h'(d(q,\T(q))) d(q,\T(q)) +  \nabla_V \nabla_V \phi(q) + \nabla_W \nabla_W \psi( \T(q)),
  \end{equation}
\begin{equation}\label{stimadersec}  \nabla_V \nabla_V \phi(q) + \nabla_W \nabla_W \psi (\T(q)) \leq 0
\end{equation}
where $V= \frac{\nabla \phi}{ \| \nabla \phi \|} $ and $W= \frac{\nabla \psi}{ \| \nabla \psi \|}$.

Let us now focus on \eqref{stima 4gi pesata}. Notice that under the assumption (cvxreg) in Setup~\ref{setup} we have $\ell(t) = g(t^2)$, $\ell'(t) = g'(t^2)2t$, $\ell''(t) = g''(t^2) 4t^2 + 2g'(t^2)$. Moreover  $\alpha = 2g'(|\nabla\phi|^2)$ and $\beta = 2g'(|\nabla\psi|^2)$.   Let $q \in \M$ be a point where $\T,\phi,\psi$ are well defined and smooth (we recall that this holds $\vol$-almost everywhere): if $\nabla \phi(q)=0$ then by (cvxreg) we have that both the right-hand side and the left-hand side in \eqref{stima 4gi pesata} are zero in a neighborhood of $q$, so we can assume that $\nabla \phi (q) \neq 0$. Developing the divergence, noticing that  $\alpha(q)=\beta(\T(q))$ and applying \eqref{stima 4gistrong} we get 
\begin{align*}
\div (\alpha \nabla \phi) (q)\,+\,  &\, \div (\beta\nabla \psi)  \big(\T(q)\big)  \\
= & \, \alpha(q) \div \nabla \phi(q) +  \nabla \alpha(q) \cdot \nabla \phi(q) + \beta(\T(q)) \div \nabla \psi(\T(q)) +  \nabla \beta(\T(q)) \cdot \nabla \psi(\T(q)) \\
\leq &  -K \alpha(q) h'(d(p,\T(q))) d(p,\T(q)) + \nabla \alpha(q) \cdot \nabla \phi(q)  + \nabla \beta(\T(q)) \cdot \nabla \psi(\T(q)) \\
  & +  \alpha(q) (\nabla_V \nabla_V \phi(q) + \nabla_W \nabla_W \psi( \T(q)))
\end{align*}

Recalling that $\alpha(q) = 2 g'(|\nabla\phi|^2)(q)$ we deduce that 
\[ \nabla \alpha(q) \cdot \nabla \phi(q) = 4 g''(|\nabla \phi|^2)(q) \nabla\phi(q) \cdot D^2\phi(q) \cdot \nabla\phi(q)  \]
and an analogous formula holds for $\nabla \beta(\T(q)) \cdot \nabla \psi(\T(q))$.  Moreover,  since $ \nabla \phi(q) = | \nabla \phi(q)| V$ and $\nabla \psi(\T(q)) = |\nabla \psi(\T(q))| W$  and $  | \nabla \phi(q)| = |\nabla \psi(\T(q))|$,  we can then estimate 
\begin{align*}
\notag
\div (\alpha \nabla \phi) (q)\,+\,  &\, \div (\beta\nabla \psi)  \big(\T(q)\big)  \\
\notag
\leq &  -K \alpha(q) h'(d(q,\T(q))) d(q,\T(q))  \\
& + \left[4 g''(|\nabla \phi|^2)(q) | \nabla \phi(q)|^2  +  2 g'(|\nabla \phi (q)|^2 \big) \right]\big( \nabla_{V} \nabla_{V} \phi(q) +  \nabla_{W} \nabla_{W} \psi(\T(q)) \big).
\end{align*}
Once here, using \eqref{stimadersec} and that $ g''(t^2) 4t^2 + 2g'(t^2) = \ell''(t) \geq 0$ we conclude using that whenever $\T(q)$ is not on the cut-locus of $q$ (that happens almost everywhere, by an argument similar to~\cite[Proposition 4.1]{Cor}), we have $\| \nabla \phi(q)\|=h'(d(q,\T(q)))$ and so $\alpha(q) h'(d(q, \T(q)))  = \ell' (h'(d(q,\T(q))))$. Hence, we can conclude.
\end{proof}

\subsection{The general case}

We are now in the position to prove the main result of this paper, namely Theorem~\ref{thm:main}.

As already anticipated, in order to get the result in full generality, dropping the smoothness of $\mu$ and $\nu$ and assumptions (cost2reg) and (cvxreg), we will perform an approximation procedure with more regular objects: we will then use Proposition~\ref{corollario a 4gi su intornino} for the smooth case and rely on Proposition~\ref{prop:stability} to conclude. We now prove the last ingredient, which is a general approximation result for $\lambda$-concave functions, which is needed in order to deal with the exceptional set $\Sigma$ which appears in Proposition~\ref{corollario a 4gi su intornino}.

\begin{prop}[Approximation of $\lambda$-concave functions]\label{prop:approxphi}
Let $\phi:  (\M,g) \to \mathbb{R}$ be an $L$-Lipschitz $\lambda$-concave function such that $\| \phi\|_{\infty} \leq M$ and $\ell$ be a convex function that satisfies (cvxreg). Suppose that there is an open set $\Omega \subseteq \M$ such that $\phi \in C^2(\Omega)$. Then for every $K \subseteq \Omega$ compact there exists $A>0$ which depends only on $\M$ and a sequence $\phi_m:  (\M,g) \to \mathbb{R}$ such that
\begin{itemize}
\item[(i)] $\phi_m \in C^2(\M)$; $\phi_m$ is also $(AL+AM)$-Lipschitz and $(A\lambda+AM+AL)$-concave; 
\item[(ii)] $\phi_m \stackrel{\M}{\to} \phi$ uniformly, $\nabla \phi_m \stackrel{L^p(\M)}{\longrightarrow} \nabla \phi $ for every $p \geq 1$ and $D^2 \phi_m \stackrel{K}{\to} D^2 \phi$ uniformly;
\item[(iii)] $\div \big( \ell' ( \nabla \phi_m ) \big) \leq  C$ distributionally in $\M$, where $C$ depends only on $L$, $\lambda$, $M$, $\M$ and $\ell$.
 \end{itemize}
\end{prop}

\begin{proof} To prove the existence of such sequence, it is sufficient to use a finite $C^2$ partition of unity $\{ \chi_i\}_{i=1}^N$  such that $|\nabla \chi_i | \leq C_1$ and $\|D^2 \chi_i\| \leq C_2$ supported on some $\{ \Omega_i\}_{i=1}^N$, then do a compactly supported convolution in charts $ \phi_{m,i}= \eta^i_m \ast (\phi \chi_i)$ and then sum up the contributions again. It is clear that $\phi_{m,i} \in C^2(\M)$ and so we have also $\phi_m \in C^2(\M)$. Moreover, by the usual properties of convolutions, we have $\phi_{m,i} \to \phi_i$ uniformly on $\M$, $\nabla \phi_{m,i} \to \nabla \phi_i$ in $L^p(\M)$, and $D^2\phi_{m,i} \to D^2\phi_i$ uniformly on $K \cap \Omega_i$, so we deduce $(ii)$ by summing up the separate contributions.
For the first point, we have
$$\nabla \phi_m = \sum_i \eta_m \ast (\chi_i \nabla \phi) + \eta_m \ast (\phi \nabla \chi_i )$$ 
 $$| \nabla \phi_m(p)| \leq L \sum_i \eta^i_m \ast \chi_i + M \sum_i \eta^i_m \ast (|\nabla \chi_i|) \leq LN + MNC_1,$$
we conclude by choosing $A \geq \max\{N, NC_1\}$. With a similar strategy, we obtain the uniform bound for the quasi-concavity using the following estimate for the distributional second derivative  
$$D^2(\phi \chi_i) = \chi_i D^2 \phi + \nabla \phi \otimes \nabla \chi_i + \phi D^2 \chi_i \leq \lambda + L C_1 + MC_2;$$
convolving and summing up the contribution we obtain the claim for the quasi-concavity by choosing $A \geq \max\{ N, NC_1, NC_2\}$

In order to do the final estimate we observe that if $\nabla \phi_m (p)=0$ there is nothing to prove since $\ell'(\nabla \phi_m)\equiv0$ in a neighborhood of $p$, while if $\nabla \phi_m (p) \neq 0$ we can consider $X_1= \nabla \phi_m / \| \nabla \phi_m\| (q)$ and $(X_i)_{i=1\ldots n}$ an orthonormal basis on $T_q \M$ complementing $X_1$: we extend this basis as geodesic normal coordinates in a small neighborhood of $q$. 

We thus have $\div(V)= \sum_i g(\nabla_{X_i} V,X_i) = \sum_i  \nabla_{X_i} g(V,X_i)$ and in particular $\div (\nabla \phi)= \sum_i \nabla_{X_i} \nabla_{X_i} \phi$: we will denote $D_{ii}\phi = \nabla_{X_i} \nabla_{X_i} \phi$, and the quasi-concavity of $\phi_m$ implies $D_{ii} \phi_m (p) \leq (A\lambda+AM+AL)$ for every $i=1,\ldots, n$. We compute now
\begin{align*}
\div \big( \ell' ( \nabla \phi_m ) \big)  &= \div \big ( 2g'(|\nabla \phi_m|^2) \nabla \phi_m \big) \\ &= 2 g'(|\nabla\phi_m|^2) \div  ( \nabla \phi_m )  + 4 g''(|\nabla\phi_m|^2) \nabla \phi_m \cdot  D^2 \phi_m \cdot \nabla \phi_m \\ 
&= 2 g'(|\nabla\phi_m|^2) \sum_i D_{ii} \phi_m + 4 g''(|\nabla\phi_m|^2) |\nabla \phi_m|^2 D_{11} \phi_m \\&= \ell''( |\nabla \phi_m|) D_{11}\phi_m + 2g'(| \nabla \phi_m|^2) \sum_{i \geq 2} D_{ii} \phi_m \\&   \leq (A\lambda+AM+AL) \sup_{[0,(AL+AM)^2]} \ell'' + 2(n-1)(A\lambda+AM+AL)g'((L+1)^2). 
\end{align*}
This concludes the proof.
\end{proof}

We have now all the ingredients to prove the main Theorem~\ref{thm:main}.

\begin{proof}[Proof of Theorem~\ref{thm:main}.] We first assume that $\mu, \nu$ are smooth and bounded below by some positive constant, $c$ satisfies (cost2reg) and $\ell$ satisfies (cvxreg). Then, we can apply the results of the previous sections, in particular, we have that $\phi, \psi$ are smooth in $\Omega= \M  \setminus (\Sigma_\mu \cup \Sigma_\nu)$. Since $\phi, \psi$ are $\lambda$-concave for some $\lambda$ and smooth in $\Omega$, we can use Proposition~\ref{prop:approxphi} applied to any compact set $D \subseteq \Omega$ to have sequences of smooth functions $\phi_m$ and $\psi_m$. In particular, we have
\begin{align*} 
 \int_{\mathcal{M}}  \ell' (\nabla \phi) \cdot \nabla \mu \vol &= \lim_{m \to \infty}  \int_{\mathcal{M}}  \ell' (\nabla \phi_m) \cdot \nabla \mu \vol \\
  & = \lim_{m \to \infty}  \int_{\mathcal{M}} \ell' (\nabla \phi_m) \cdot \nabla \mu  \vol  \\
 & =  - \lim_{m \to \infty}  \int_{\mathcal{M}} \div \big(  \ell' (\nabla \phi_m) \big)\  \mu  \vol \\
 & \geq   -C\mu(\mathcal{M} \setminus D)  -\lim_{m \to \infty}  \int_{D} \div \big(  \ell' (\nabla \phi_m) \big)\  \mu \vol \\
 & = - C\mu(\mathcal{M} \setminus D) - \int_{D} \div \big(  \ell' (\nabla \phi) \big) \, d \mu
 \end{align*}
Similarly, we can deduce
$$ \int_{\mathcal{M}}  \ell' (\nabla \psi) \cdot \nabla \nu \vol \geq  - C\nu(\mathcal{M} \setminus \T (D)) - \int_{\T(D)} \div \big(  \ell' (\nabla \psi) \big)  \, d \nu. $$
We can now consider $D$ big enough such that $\mu(\mathcal{M} \setminus D) + \nu(\mathcal{M} \setminus \T (D)) \leq \eps$ and summing up the previous inequalities we get
\begin{align*} 
\int_{\M}  \Bigl(  \ell' (\nabla \phi) \cdot \nabla \mu +   \ell' (\nabla \psi) \cdot \nabla \nu  \Bigr) \vol   &\geq - C \eps - \int_{D} \div \big(  \ell' (\nabla \phi) \big) \, d \mu - 
\int_{\T(D)} \div \big(  \ell' (\nabla \psi) \big)  \, d \nu \\
& = - C \eps - \int_D   \Bigl(  \div \big(  \ell' (\nabla \phi (q)) \big) +  \div \big(  \ell' (\nabla \psi( \T(q)) \big)  \Bigr) \, d \mu(q).
 \end{align*}
We can  now apply Proposition~\ref{corollario a 4gi su intornino} for every point $q \in D \subseteq \M \setminus \Sigma$ to deduce 
$$ \int_{\M}  \Bigl(  \ell' (\nabla \phi) \cdot \nabla \mu +   \ell' (\nabla \psi) \cdot \nabla \nu  \Bigr) \vol  \geq - C \eps + K \int_D \ell'(h'(d(q,\T(q))) d(q,\T(q))\, d\mu(q).$$
Letting now $D \uparrow \Omega$ we have $\eps \to 0$ and we get the conclusion by monotone convergence and the fact that $\mu(\M\setminus\Omega)=0$.

In order to prove \eqref{eqn:5GIW11} for any $\mu, \nu \in W^{1,1}(\M)$, we perform another approximation argument: we consider $\mu_m, \nu_m$ smooth positive densities which converge in $L^1$ to $\mu, \nu$ respectively and such that $\nabla \mu_m \to \nabla \mu$ and $\nabla \nu_m \to \nabla\nu $ in $L^1(\M)$, where we can assume that the convergence is dominated. Moreover, thanks to Proposition \ref{prop:stability} we also have that $\nabla \phi_m \to \nabla \phi$ pointwise $\vol$-almost everywhere, where $\phi_m$ are the Kantorovich potentials for $\mu_m,\nu_m$, and so we can pass to the limit via dominated convergence since we also have $\phi_m, \phi$ uniformly Lipschitz. In order to prove the theorem also for $c$ which satisfies (cost), we can approximate it uniformly with $c_m$ which satisfies (cost2reg) and again we conclude using Proposition \ref{prop:stability}.
Finally, we can approximate every $\ell$ satisfying (cvx) with some $\ell_m$ satisfying (cvxreg) in such a way that $\ell_m' \uparrow \ell'$ and we conclude again by dominated convergence (notice that this approximation argument works also when $\ell'(0) \neq 0$, paying attention to Remark~\ref{rmk:ell} for the correct interpretation of \eqref{eqn:5GIW11}).
\end{proof}



\end{document}